\documentclass[final,1p,times]{elsarticle}
\usepackage{amsmath}
\usepackage{amsthm}
\usepackage{amssymb}
\usepackage{amsfonts}
\newtheorem{theorem}{Theorem}

\newtheorem{remark}[theorem]{Remark}

\biboptions{sort&compress} 
\usepackage{verbatim} 
\numberwithin{equation}{section} 
\numberwithin{theorem}{section} 
\usepackage{hyperref} 
\makeatletter \renewcommand*{\eqref}[1]{\hyperref[{#1}]{\textup{\tagform@{\ref*{#1}}}}} \makeatother 
\usepackage{array}   
	\newcolumntype{L}{>{$}l<{$}} 
	\newcolumntype{R}{>{$}r<{$}} 
	\newcolumntype{C}{>{$}c<{$}} 

\makeatletter
\def\ps@pprintTitle{%
   \let\@oddhead\@empty
   \let\@evenhead\@empty
   \let\@oddfoot\@empty
   \let\@evenfoot\@oddfoot
}
\makeatother

\begin{document}

\begin{frontmatter}

\title{On Polynomial Solutions of Linear Differential Equations with Applications}

\author{Kyle R. Bryenton$^{1,*}$} \ead{kyle.bryenton@dal.ca} 
\author{Andrew R. Cameron$^{2}$} \ead{ar3camer@uwaterloo.ca}
\author{Keegan L. A. Kirk$^{1}$} \ead{k4kirk@uwaterloo.ca}
\author{Nasser Saad$^{1}$} \ead{nsaad@upei.ca}
\author{Patrick Strongman$^{2}$} \ead{patrick.strongman@dal.ca}
\author{Nikita Volodin$^{1}$} \ead{nikita.volodin@uwaterloo.ca}
\address{$^{1}$ School of Mathematical and Computational Sciences, University of Prince Edward Island, Charlottetown, PE C1A 4P3, Canada}
\address{$^{2}$ Department of Physics, University of Prince Edward Island, Charlottetown, PE C1A 4P3, Canada}

\begin{abstract}
The analysis of many physical phenomena can be reduced to the study of solutions of differential equations with polynomial coefficients. In the present work, we establish the necessary and sufficient conditions for the existence of polynomial solutions to the linear differential equation  
\begin{equation*}
\sum_{k=0}^{n} \alpha_{k} \, r^{k} \, y''(r) + \sum_{k=0}^{n-1}  \beta_{k} \, r^{k} \, y'(r) - \sum_{k=0}^{n-2} \tau_{k} \, r^{k} \, y(r) = 0 \, ,
\end{equation*} 
for arbitrary $n\geq 2$. We show by example that for $n \ge 3$, the necessary condition is not enough to ensure the existence of the polynomial solutions. Using Scheff\'{e}'s criteria, we show that from this differential equation there are $n$-generic equations solvable by a two-term recurrence formula. The closed-form solutions of these generic equations are given in terms of the generalized hypergeometric functions. For the arbitrary $n$ differential equations, three elementary theorems and one algorithm were developed to construct the polynomial solutions explicitly. The algorithm is used to establish the polynomial solutions in the case of $n=4$. To demonstrate the simplicity and applicability of this approach, it is used to study the solutions of Heun and Dirac equations. 
\end{abstract}

\begin{keyword}
Symbolic computation; Polynomial solutions; Scheff\'{e} criteria; Heun equation; Dirac equation; Method of Frobenius; Recurrence relations.\\
\vspace*{0.2cm}
\end{keyword}

\end{frontmatter}

\section{Introduction} \label{Sec1}
Differential equations with polynomial coefficients have played an important role not only in understanding engineering and physics problems \cite{sa2014,zh2012,wk2014,ra2007,mh2011,ci2010,mf2018,ch2014,ch2015,ch2016,ish2015,ish2018,ish2018b,li2016,fe2017,ch2002,ca2014,ha2010,ma1985,ha2011,ha2011b,ov2012,ex1991,qd2014,qd2016,qd2018,qd2018a,qd2019,qd2019a}, but also as a source of inspiration for some of the most important results in special functions and orthogonal polynomials 
\cite{eu1769,ro1884,bo1929,sc1941,ss1964,ha1978,li1982,ll1984,sa1990,at1995,kl1998,ek2001,ro1995b,ma2007,ja2008,bl2012}. The classical differential equation
\begin{equation}\label{eq1.1}
\big( \alpha_{2} \, r^{2} + \alpha_{1} \, r + \alpha_{0} \big) \, y''(r) + \big( \beta_{1} \, r + \beta_{0} \big) \, y'(r) - \tau_{0} \, y(r) = 0 \, , 
\end{equation}
for example, has been a cornerstone of many fundamental results in special functions since the early work of L. Euler \cite{eu1769}, see also \cite{sw2000,ro1884,bo1929}. More recently, the Heun-type equation \cite{ro1995b,ma2007,ja2008}
\begin{equation}\label{eq1.2}
r \, \big( \alpha_{3}\, r^{2}+ \alpha_{2} \, r + \alpha_{1} \big) \, y''(r) + \big( \beta_{2}\, r^{2}+ \beta_{1} \, r + \beta_{0} \big) \, y'(r) - \big( \tau_{1}\, r+ \tau_{0} \big) \, y(r) = 0 \, , 
\end{equation}
has become a classic equation in mathematical physics. These two equations are members of the general differential equation
\begin{align}\label{eq1.3}
&P_{n}(r) \, y''(r) + Q_{n-1}(r) \, y'(r) - R_{n-2}(r) \, y(r) = 0 \, , \qquad n\geq 2 \, , \notag \\
&P_{n}(r)=\sum_{k=0}^{n} \alpha_{k}\, r^{k} \, , \quad Q_{n-1}(r)= \sum_{k=0}^{n-1} \beta_{k}\, r^{k} \, , \quad R_{n-2}(r)=\sum_{k=0}^{n-2} \tau_{k}\, r^{k} \,,
\end{align}
where the leading term of $R_{n-2} \neq 0$ and at least one of the leading terms $\alpha_{n}$ and $\beta_{n-1}$ is not zero. The purpose of the present work is to provide an answer to the following question:

\vspace{12pt}

\textit{``Under what conditions on the equation parameters $\alpha_{k}$, $\beta_{k}$, and $\tau_{k}$, $k=0,\,1,\,\ldots,\,n$, does the differential equation \eqref{eq1.3}
 have polynomial solutions? If it does have polynomial solutions, how can we construct them?''}\\

With the general theory of linear differential equations as a guide \cite{in1956,fo1933}, the zeros of the leading polynomial $P_{n}(r)$ classify the solutions of the equation \eqref{eq1.3}. The point $r = \zeta$ is called an ordinary point if $Q_{n-1}(r) / P_{n}(r)$ and $R_{n-2}(r) / P_{n}(r)$ are analytic functions at $r = \zeta$, or a regular singular point if $Q_{n-1}(r) / P_{n}(r)$ and $R_{n-2}(r) / P_{n}(r)$ are not analytic at $r = \zeta$ but the products $(r-\zeta) \, Q_{n-1}(r) / P_{n}(r)$ and $(r-\zeta)^{2} \, R_{n-2}(r) / P_{n}(r)$ are analytic at that point. Using this characterization of singularities, S. Bochner \cite{bo1929} classified the polynomial solutions of \eqref{eq1.1} in terms of the classical orthogonal polynomials. A different approach which depends on the parameters of the leading coefficient $P_{n}(r)$, was introduced in \cite{sa2014} to study the polynomial solutions of \eqref{eq1.1}. It was shown that there are seven possible nonzero leading polynomials depending on the combination of the nonzero $\alpha_{i}$ parameters, for $i=0,\,1,\,2$. By analyzing each of these cases, the authors \cite{sa2014} were able to explicitly construct all the polynomial solutions of equation \eqref{eq1.1} in terms of hypergeometric functions along with the associated weight functions. 
\vskip0.1true in
For the $n^{th}$ degree leading polynomial coefficient $P_{n}(r)$ of the differential equation \eqref{eq1.3} contains $(2^{n+1}-1)$ nonzero polynomials that depend on the nonzero values of $\alpha_{i}$, $i=0,\,1,\,\ldots\,,n$. Out of this polynomial set, there are $2^{n}$ differential equations for which $r=0$ is an ordinary point, $(2^{n-1} +2^{n-2})$ differential equations for which $r=0$ is a regular singular point, and the remaining $(2^{n-2}-1)$ differential equations with irregular singular points that fall outside of the scope of this present work.\\

The criteria for polynomial solutions of second-order linear differential equation \eqref{eq1.3} was introduced, using the Asymptotic Iteration Method, in \cite{sa2003,sa2005}:
\begin{theorem}\label{Thm1} \cite{sa2003,sa2005} Given $\lambda_{0}\equiv\lambda _{0}(x)$ and $s_{0}\equiv s_{0}(x)$ in $C^{\infty}$, the differential equation $y''=\lambda_{0}(x) y'+s_{0}(x) y$ has the general solution
\begin{equation} \label{eq1.4}
y=\exp \left( -\int\limits^{x}\alpha (t)dt\right) \left[ C_{2}+C_{1}\int\limits^{x}\exp \left( \int\limits^{t}\left( \lambda _{0}(\tau )+2\alpha(\tau )\right) d\tau \right) dt\right]
\end{equation}
if, for some positive integer $m>0$, 
\begin{equation} \label{eq1.5}
\frac{s_{m}}{\lambda _{m}}=\frac{s_{m-1}}{\lambda _{m-1}}=\alpha (x) \, , \quad \text{or} \quad \delta _{m}(x)=\lambda _{m}s_{m-1}-\lambda _{m-1}s_{m}=0 \, ,
\end{equation}
where $\lambda _{m}=\lambda _{m-1}^{\prime }+s_{m-1}+\lambda_{0}\lambda _{m},$ and $s_{m}=s_{m-1}^{\prime }+s_{0}\lambda _{m}$. Further, the differential equation has a polynomial solution of degree $m$, if 
\begin{equation} \label{eq1.6}
\delta _{m}(x)=\lambda _{m}s_{m-1}-\lambda _{m-1}s_{m}=0.
\end{equation}
\end{theorem}
With $\lambda_{0}(r)= - Q_{n-1}(r)/P_{n}(r)$ and $s_{0}(r)= R_{n-2}(r)/P_{n}(r)$, a simple algorithm based on Theorem \ref{Thm1} can be used to examine the polynomial solutions of \eqref{eq1.3}. 
\vskip0.1true in
\noindent The present work independently from theorem \ref{Thm1} establish these polynomial solutions along with the required conditions for the existences of these solutions in simple and constructive forms. The paper is organized as follows: Section 2 introduces the concept of a necessary but not sufficient condition to find polynomial solutions, and then demonstrates how the inverse square-root potential \cite{li2016,fe2017, ish2018, ish2018b} is one such example that does not have polynomial solutions yet satisfies the necessary condition. In Section 3, Scheff\'{e} criteria is revised to analyze equation \eqref{eq1.3} and to generate exactly solvable equations where the coefficients of their power series solutions are easily computed using a two-term recurrence relation. The closed form solutions of these equations in terms of the generalized hypergeometric functions are also given. In Section 4 we present three theorems corresponding to the cases: (1) $\alpha_{0} \neq 0$, (2) $\alpha_{0} = 0$, $\alpha_{1} \neq 0$, and (3) $\alpha_{0} = \alpha_{1} = \beta_{0} = 0$, $\alpha_{2} \neq 0$ which are required to establish the polynomial solutions of equation \eqref{eq1.3}. We shall show that for the $m^{th}$-degree polynomial solutions, there are $n+m-1$ conditions that ultimately assemble these polynomials. At the end of the section we briefly discuss the Mathematica\textsuperscript{\textregistered} program used to generate the solutions of these equations. Lastly, Section 5 demonstrates the validity of these constructions through the application of our results to some problems that have appeared in mathematics and physics literature. \\

\section{A Necessary But Not Sufficient Condition} \label{Sec2}
We begin by stating the necessary condition for the existence of polynomial solutions to the general differential equation given by \eqref{eq1.3}:
\begin{theorem}\label{Thm2.1}
A necessary condition for the differential equation \eqref{eq1.3} to have a polynomial solution of degree $m$ is
\begin{equation}\label{eq2.1}
\tau_{n-2} = \alpha_{n} \, m \, (m-1) + \beta_{n-1} \, m \, , \qquad n = 2 , \, 3 , \, \ldots \, .
\end{equation}
\end{theorem}
\begin{proof}
Substitute the polynomial solution $y=\sum_{i=0}^{m} \mathcal{C}_{i} \, r^{i}$ into the differential equation \eqref{eq1.3}. After collecting the common terms of the common $r$ power, the necessary condition for the leading coefficient $\mathcal{C}_{m}$ to be nonzero is $\tau_{n-2} = \alpha_{n} \, m \, (m-1) + \beta_{n-1} \, m$ for all $m\geq 0$ follows.
\end{proof}

For $n=2$, equation \eqref{eq2.1} represents both the necessary and sufficient conditions for the polynomial solutions of \eqref{eq1.3}. Although the condition \eqref{eq2.1} is necessary, it is not sufficient for $n \ge 3$ to guarantee the polynomial solutions of \eqref{eq1.3}. There are precisely $n-2$ additional conditions that are sufficient to ensure the existence of such polynomials. These additional conditions can be understood as constraints that relate the remaining parameters of $R_{n-2}(r)$ with the coefficients $\{ \alpha_{k} \}_{k=0}^{n}$ and $\{ \beta_{k} \}_{k=0}^{n-1}$ in $P_{n}(r)$ and $Q_{n-1}(r)$, respectively. In later sections, we shall devise a procedure to find these sufficient conditions and provide a method of evaluating the corresponding polynomial solutions explicitly. \\

The remainder of the section investigates the inverse square-root potential. We shall show that despite satisfying the necessary condition in Theorem \ref{Thm2.1} that the Schr\"{o}dinger equation with this potential as mentioned in the literature has no polynomial solutions, thus demonstrating that the condition \eqref{eq2.1} is necessary but not sufficient. 
\subsection{The Inverse Square-Root Potential}
The exact solutions of the radial Schr\"{o}dinger equation
\begin{equation}\label{eq2.2}
-\dfrac{d^{2} \, \psi(r)}{dr^{2}}+\left(-\frac{v}{\sqrt{r}}+\dfrac{l \, (l+1)}{r^{2}}\right) \psi(r) = E \, \psi(r) \, , \qquad v>0 \, , \quad r>0 \, , \quad \psi(0)=\psi(\infty)=0 \, ,
\end{equation}
were recently studied by a number of authors \cite{ish2015,li2016,fe2017,ish2018,ish2018b}. The differential equation \eqref{eq2.2} has two singular points, one regular at $r=0$ and another irregular at $r=\infty$ of rank $3$ \cite{in1956}. To analyze these points further, we employ a change of variables $z=\sqrt{2 \,\kappa \, r}$, for some constant $\kappa$ to be determined shortly. Equation \eqref{eq2.2} may then be expressed as
\begin{equation}\label{eq2.3}
z^{2} \,\psi''(z)-z\,\psi'(z)+\left(\left(\frac{\sqrt{2} \, v}{\kappa^{3/2}}+\frac{E}{\kappa^{2}}z\right) z^{3}-4 \, l \, (l+1)\right)\psi(z)=0 \, , \qquad \psi(0)=\psi(\infty)=0 \, .
\end{equation}
As $z\to 0$, the solution to the differential equation \eqref{eq2.3} behaves asymptotically as the solution of the Euler equation, $z^{2} \, \psi''(z) - z \, \psi'(z) - 4 \, l \, (l+1) \, \psi(z) = 0 $, which has a physically acceptable solution $\psi(z) \sim z^{2 \, (l+1)}$. Meanwhile as $z \to \infty$, the solution to the differential equation behaves asymptotically as the solution of $\psi''(z)+\left(\sqrt{2} \, v / \kappa^{3/2} - z\right) \, z \, \psi(z) = 0$.
To examine the solution of this equation we complete the square that yields
\begin{equation}\label{eq2.4}
-\psi''(z)+\left(\frac{v}{\sqrt{2 \, \kappa^{3}}}-z\right)^{2} \, \psi(z)=\left(\frac{v}{\sqrt{2 \, \kappa^{3}}} \right)^{2} \, \psi(z) \,.
\end{equation}
The substitution $x=\frac{v}{\sqrt{2 \, \kappa^{3}}}-z$ allow us to compare the solution of the equation to the solvable Schr\"odinger equation with the harmonic oscillator potential, as $z\to\infty$, via
\begin{equation}\label{eq2.5}
\psi(z)=\exp\left(-\int^{z} (u+\lambda) \, du \right) = \exp\left(-\left(\frac{z^{2}}{2}+\lambda \, z\right)\right),\quad \lambda= -\frac{v}{\sqrt{2 \, \kappa^{3}}} \, .
\end{equation}
Hence, we may assume the solution of equation \eqref{eq2.3} to be
\begin{equation}\label{eq2.6}
\psi(z) = N \, z^{2 \,(l+1)}\,\exp\left(-\left(\dfrac{z^{2}}{2}+\lambda z\right)\right)\,f(z) \, , \qquad \lambda= -\frac{v}{\sqrt{2 \, \kappa^{3}}}<0 \, ,
\end{equation}
up to a normalization constant $N$. Upon substituting the ansatz solution \eqref{eq2.6} into equation \eqref{eq2.3}, we obtain a second-order differential equation (an example of a biconfluent Heun equation \cite{ro1995b}) for $f(z)$ as
\begin{equation}\label{eq2.7}
z \, f''(z)-\left[2 \, z^{2} + 2 \, \lambda \, z-3-4 \, l \, \right] f'(z) + \Big[ \left(\lambda^{2}-4 \, (l+1) \right) z - (3+4 \, l) \, \lambda \Big] \, f(z) = 0 \, ,
\end{equation}
subject to $ \kappa = \sqrt{-E}$. Clearly, equation \eqref{eq2.7} is of the form of equation \eqref{eq1.3} with $n = 3$ given $\alpha_{3} = \alpha_{2} = 0$, $\alpha_{1} = 1$, $\beta_{2} = -2$, $\beta_{1} = -2 \, \lambda$,  $\beta_{0} = 3+4 \, l$, $\tau_{1} = -\lambda^{2}+4 \,(l+1)$, and $\tau_{0} = (3+4 \, l) \, \lambda$. Using Theorem \ref{Thm2.1}, for the existence of an $m^{th}$-degree polynomial solution it is necessary that 
\begin{equation}\label{eq2.8}
\lambda^{2}- 4 \, (l+1) = 2\,m \, , \qquad m = 0, \, 1, \, 2, \, \ldots \, .
\end{equation}
Note that this result can be deduced using Bethe Ansatz method \cite{zh2012,shi2011}. On the other hand, an application of the Frobenius method establishes the following three-term recurrence relation for the coefficients of the polynomial solutions $f_{m}(r) = \sum_{j=0}^{m} \mathcal{C}_{j} \, r^{j}$:
\begin{equation}\label{eq2.9}
(j+1) \left(j+3+4 \, l \right)  \mathcal{C}_{j+1} - \lambda \left(2 \, j+3+4 \, l \right) \mathcal{C}_{j} - \left(2 \, j- \lambda^{2}+4 \, l+2 \right) \mathcal{C}_{j-1} = 0 \, ,
\end{equation} 
where $j = 0, \, 1, \, 2, \, \ldots ,\, m, \, m+1$ and $C_{j}=0$ for $j<0$. Implementing the necessary condition \eqref{eq2.8}, with $\lambda=-\sqrt{4+4l+2m}<0$, equation \eqref{eq2.9} then reads  
\begin{equation}\label{eq2.10}
(j+1)\left(j+3+4 \, l\right) \mathcal{C}_{j+1} +\sqrt{4+4 l+2m} \, \big(2 \, j+3+4 \, l \,\big)\, \mathcal{C}_{j} +2 \big( m- \, j+1 \big) \, \mathcal{C}_{j-1} = 0 \, .
\end{equation}
The three-term recurrence relation given by equation \eqref{eq2.10} enforces the vanishing of the determinant, denoted by $\Delta_{(m+1)}$, that gives the sufficient condition
\begin{center}
\begin{minipage}{0.494\textwidth}
\begin{equation*}
\Delta_{m+1} = \left| \begin{array}{ccccccc}
S_{0} & T_{1} &&&&& \\
\gamma_{1} & S_{1} & T_{2} &&&& \\
& \gamma_{2} & S_{2} & T_{3} &&& \\
&& \ddots & \ddots & \ddots &&\\ 
&&& \gamma_{m-2} & S_{m-2} & T_{m-1} & \\
&&&& \gamma_{m-1} & S_{m-1} & T_{m} \\
&&&&& \gamma_{m} & S_{m} \\
\end{array} \right| \, ,
\end{equation*}
\end{minipage}
\begin{minipage}{0.5\textwidth}
\begin{align}
\text{for} \notag\\
S_{j} &= (2 \, j+3+4 \, l)\, \sqrt{4+4 l+2m} \, , \notag\\
T_{j} &= j(j+2+4 \, l) \, , \notag\\
\gamma_{j} &= 2(m-j+1) \, . \label{eq2.11}
\end{align}
\end{minipage}
\end{center}
For any nonzero value of the determinant $\Delta_{m+1}$, it will be  indication that a polynomial solution is not possible.

\begin{remark} The point here is not solving the Schr\"{o}dinger equation with the inverse-root potential; this is already done in several manuscripts \cite{ish2015,li2016,fe2017,ish2018,ish2018b}. We stress that the necessary condition is not enough to guarantee polynomial solutions. Also, the necessary and sufficient constraints must have common roots; however, at least some of these roots must belong to the physically interpreted region.
\end{remark}

\begin{remark} In \cite{ro1974}, the author claims the existence of polynomial solutions of a differential equation which matches \eqref{eq2.7} that satisfy the necessary condition \eqref{eq2.1}. The author missed the possibility of the non-vanishing determinant (4), in his work, for some $\beta_{0}$ and $\beta_{2}$.
\end{remark}

\begin{remark} For physics applications, there is nothing special about the necessity of the polynomial solutions. However, the problem of the existence of polynomial solutions is important. Studying the problem in its general form and developing mathematical tools to treat it is significant on its own.
\end{remark}

\section{Scheff\'{e}'s Criteria: Two-Term Recurrence Relation} \label{Sec3}
Generally speaking, recurrence relations with more than two terms are difficult to solve explicitly. Differential equations that are known to have a two-term recurrence relation for their power series solutions guarantee the solvability of such equations. In a paper presented to the American Mathematical Society (1941), H. Scheff\'{e} \cite{sc1941} establishes criteria for the necessary and sufficient conditions for differential equations of the form
\begin{equation}\label{eq3.1}
p_{2}(r)\, u''(r)+p_{1}(r)\, u'(r)+p_{0}(r)\, u(r)=0 \, ,
\end{equation}
to have a two-term recurrence relation. Here $p_{j}(r)\neq 0,~j=0,1,2,$ are analytic functions (not necessarily polynomials) in some region consisting of all points in an arbitrary neighbourhood of a regular singular point $r = r_{0}$ except the point $r_{0}$ itself. We adopt this criterion to examine the differential equation \eqref{eq3.1} with 
\begin{equation} \label{eq3.2}
p_{2}(r)=\sum_{k=0}^{n} \alpha_{k}\, r^{k} \, , \quad p_{1}(r)= \sum_{k=0}^{n-1} \beta_{k}\, r^{k} \, , \quad p_{0}(r)=\sum_{k=0}^{n-2} \gamma_{k}\, r^{k} \, , \quad n\geq 2 \, , 
\end{equation}
and provide a formula for the two-term recurrence relation. Without loss of generality, we shall take the ordinary or singular point as $r=0$ otherwise a simple shifting of $r$ is applied first.

\begin{theorem} \label{Thm3.1} The necessary and sufficient conditions for the  differential equation \eqref{eq3.1} to have a two-term recurrence relationship between successive coefficients in its series solution, relative to the ordinary or regular singular point $r=0$, is that in the neighbourhood of $r=0$ the equation \eqref{eq3.1} can be written as:
\begin{align}\label{eq3.3}
[q_{2,0}+q_{2,h}\,r^h]&\,r^{2-m}\, u''(r)+[q_{1,0}+q_{1,h}\,r^h]\,r^{1-m}\, u'(r)+[q_{0,0}+q_{0,h}\,r^h]\, r^{-m}\, u(r)=0 \, , 
\end{align}
where, for $m\in \mathbb Z$ , $h\in \mathbb Z^+$,
\begin{equation}\label{eq3.4}
q_{j}(r) \equiv \sum_{k=0}^\infty q_{j,k}r^{k}= p_{j}(r)\,r^{m-j} \, , \qquad j=0,1,2\, ,
\end{equation}
and at least one of the product $q_{j,0} \, q_{i,h}$, for $i,\,j=0,\,1,\,2 \ldots$, is not zero.  The two-term recurrence formula is given by
\begin{align}\label{eq3.5}
c_{k}=- \dfrac{q_{2,h}\, (k+\lambda-h)(k+\lambda-h-1)+q_{1,h}\, (k+\lambda-h) +q_{0,h}}{q_{2,0}\, (k+\lambda)(k+\lambda-1)+q_{1,0}\, (k+\lambda) +q_{0,0}}c_{k-h} \, ,\qquad c_{0}\neq 0 \, , 
\end{align}
where $\lambda=\lambda_{1} \, ,\lambda_{2}$ are the roots of the indicial equation
\begin{equation}\label{eq3.6}
q_{2,0}\,\lambda\,(\lambda-1)+q_{1,0}\, \lambda +q_{0,0}=0 \, .
\end{equation}
The closed form of the series solution can be written in terms of the generalized hypergeometric function as
\begin{align}\label{eq3.7}
u&(r;\lambda)=z^\lambda \sum_{k=0}^\infty c_{hk}r^{hk}\notag\\
&=r^\lambda\, _{3}F_{2}\left(1,\frac{2\lambda-1}{2h}+\frac{q_{1,h}}{2\,h\,q_{2,h}}-\frac{\sqrt{(q_{1,h}-q_{2,h})^{2}-4 q_{0,h} q_{2,h}}}{2\, h\,q_{2,h}} \, ,\right.\frac{2\lambda-1}{2h}+\frac{q_{1,h}}{2\,h\,q_{2,h}}+\frac{\sqrt{(q_{1,h}-q_{2,h})^{2}-4 q_{0,h} q_{2,h}}}{2\, h\, q_{2,h}};\notag\\
&1+\frac{2 \lambda-1}{2\, h}+\frac{q_{1,0}}{2\, h\, q_{2,0}}-\frac{\sqrt{(q_{1,0}-q_{2,h})^{2}-4 q_{0,0} q_{2,0}}}{2\, h\, q_{2,0}} \, ,\left. 1+\frac{2 \lambda-1}{2 h}+\frac{q_{1,0}}{2\, h\, q_{2,0}}+\frac{\sqrt{(q_{1,0}-q_{2,h})^{2}-4 q_{0,0} q_{2,0}}}{2\, h\, q_{2,0}};-\frac{q_{2,h}}{q_{2,0}}\,r^h\right).
\end{align}
\end{theorem}
Consider the differential equation
\begin{align}\label{eq3.8}
(\alpha_{2}\, r^{2}+\alpha_{1}\, r+\alpha_{0})\, u''(r)+(\beta_{1}\, r+\beta_{0})\, u(r)+\gamma_{0}\, u(r)=0 ,
\end{align}
mentioned in the introduction. We want to extract the possible equations with a two-term recurrence relationship between successive coefficients in its series solution. Comparing equation \eqref{eq3.8} with \eqref{eq3.3}, we find that $h+2-m=2$ and $h+2-m=2 >2-m\geq 0$, which leads to $h=m$, $2>2-m\geq 0$. Consequently  $0< m\leq 2$ which gives us two cases: $(m,h)=(1,1)$, and $(m,h)=(2,2)$. \\

For $(m,h)=(1,1)$, equation \eqref{eq3.3} reads: 
\begin{equation}\label{eq3.9}
[q_{2,0}+q_{2,1} r]\, r\, u''(r)+[q_{1,0}+q_{1,1}r]\,  u'(r)+[q_{0,0}+q_{0,1}r]\, r^{-1} \, u(r)=0 \, ,
\end{equation}
and the corresponding equation \eqref{eq3.8} with $q_{2,0}=a_{1}$, $q_{2,h=1}=a_{2}$, $q_{1,0}=b_{0}$, $q_{1,h=1}=b_{1}$, $q_{0,h=1}=\gamma_{0}$, reads
\begin{equation}\label{eq3.10}
(\alpha_{1}z+\alpha_{2} z^{2}) \, u''(z)+(\beta_{0}+\beta_{1}z)\,  u'(z)+\gamma_{0}\, u(z)=0 \, .
\end{equation}
While for $(m,h)=(2,2)$, equation \eqref{eq3.3} reads
\begin{equation}\label{eq3.11}
[q_{2,0}+q_{2,h}\, r^{2}]\, u''(r)+[q_{1,0}+q_{1,h}r^{2}]\, r^{-1}\,  u'(r)+[q_{0,0}+q_{0,h}\, r^{2}]\, r^{-2}\, u(r)=0 \, ,
\end{equation}
and the corresponding equation \eqref{eq3.8} reads
\begin{equation}\label{eq3.12}
(a_{0}+a_{2}\, r^{2})\, u''(r)+b_{1}\,r\, u'(r)+\gamma_{0}\, u(r)=0 \, .
\end{equation}
The two-term recurrence formula for the solution of equation \eqref{eq3.10} is
 \begin{equation}\label{eq3.13}
c_{k}=- \dfrac{ a_{2}\, (k+\lambda-1)(k+\lambda-2)+b_{1}\, (k+\lambda-1) +\gamma_{0} }{ a_{1}\, (k+\lambda)(k+\lambda-1)+b_{0}\, (k+\lambda) }c_{k-1} \, ,\qquad c_{0}=1 \, ,
\end{equation}
where $\lambda$ are the zeros of the indicial equation $a_{1}\lambda\,(\lambda-1)+b_{0}\, \lambda=0$, namely, $\lambda=0$, and $\lambda = 1-{b_{0}}/{a_{1}}$. The closed form solutions in terms of the Gauss hypergeometric functions, respectively, are:
\begin{equation}\label{eq3.14}
u_{1}(r,\lambda=0)=\, _{2}F_{1}\left(-\frac12+\frac{b_{1}}{2a_{2}}-\frac{\sqrt{(a_{2}-b_{1})^{2}-4 a_{2} \gamma_{0}}}{2 a_{2}} \, ,
-\frac12+\frac{b_{1}}{2a_{2}}+\frac{\sqrt{(a_{2}-b_{1})^{2}-4 a_{2} \gamma_{0}}}{2 a_{2}}; \frac{b_{0}}{a_{1}};-\frac{a_{2}}{a_{1}}\,r\right) \, ,
\end{equation}
and
\begin{align}\label{eq3.15}
u_{2}(r,\lambda&=1-\tfrac{b_{0}}{a_{1}})=r^{1-\tfrac{b_{0}}{a_{1}}}\notag\\
&\times {\,}_{2}F_{1}\left(\frac12+\frac{b_{1}}{2a_{2}}-\frac{b_{0}}{a_{1}}+\frac{\sqrt{(a_{2}-b_{1})^{2}-4 a_{2} \gamma_{0}} }{2 a_{1} a_{2}} \, ,
\frac12+\frac{b_{1}}{2a_{2}}+\frac{b_{0}}{a_{1}}+\frac{\sqrt{(a_{2}-b_{1})^{2}-4 a_{2} \gamma_{0}} }{2 a_{1}a_{2}};2-\frac{b_{0}}{a_{1}};-\frac{a_{2}}{a_{1}}\,r\right) \, .
\end{align}
The two-term recurrence relation for the solutions of equation \eqref{eq3.12} reads
\begin{align}\label{eq3.16}
c_{k}=- \dfrac{a_{2}\, (k+\lambda-2)(k+\lambda-3)+b_{1}\, (k+\lambda-2) +\gamma_{0}}{ a_{0}\, (k+\lambda)(k+\lambda-1)}\, c_{k-2} \, , \quad k\geq 2 \, ,
\end{align}
where, $\lambda$ are the roots of the indicial equation $\lambda\, (\lambda-1)=0$, $a_{0}\neq 0$, with closed form solutions
\begin{equation}\label{eq3.17}
u_{1}(r,\lambda=0)=\, _{2}F_{1}\left(\frac{b_{1}}{4a_{2}}-\frac{1}{4}-\frac{\sqrt{(a_{2}-b_{1})^{2}-4 a_{2}\gamma_{0}}}{4 a_{2}} \, ,\frac{b_{1}}{4a_{2}}-\frac{1}{4}+\frac{\sqrt{(a_{2}-b_{1})^{2}-4 a_{2} \gamma_{0}}}{4 a_{2}};\frac{1}{2};-\frac{a_{2}}{a_{0}}\, r^{2}\right) \, , 
\end{equation}
and
\begin{equation}\label{eq3.18}
u_{2}(r,\lambda=1)=r \, _{2}F_{1}\left(\frac{b_{1}}{4a_{2}}+\frac{1}{4}-\frac{\sqrt{(a_{2}-b_{1})^{2}-4 a_{2} \gamma_{0}}}{4 a_{2}} \, ,\frac{b_{1}}{4a_{2}}+\frac{1}{4}+\frac{\sqrt{(a_{2}-{b_{1}})^{2}-4 {a_{2}} {\gamma_{0}}}}{4 a_{2}};\frac{3}{2};-\frac{a_{2}}{a_{0}}r^{2}\right) \, .
\end{equation}
\begin{remark}  
For the admissible values of the equation parameters, the recurrence relations \eqref{eq3.13} and \eqref{eq3.16} are generic formulas for the two-term recurrence relations for the series solutions of the differential equations \eqref{eq3.10} and \eqref{eq3.12}. In a sense that by assigning admissible values of the parameters, the differential equations \eqref{eq3.10} and \eqref{eq3.12} generate solvable differential equations.
\end{remark} 
\begin{remark}  
Polynomial solutions for the differential equations \eqref{eq3.10} and \eqref{eq3.12} can be easily obtained using suitable values of the equation parameters to terminate the recurrence relations \eqref{eq3.13} and \eqref{eq3.16}.
\end{remark} 

As a second example we consider the differential equation
\begin{equation} \label{eq3.19}
\big( \alpha_{3} \, r^{3} + \alpha_{2} \, r^{2} + \alpha_{1} \, r +\alpha_{0}\big) \, y''(r) + \big( \beta_{2} \, r^{2} + \beta_{1} \, r + \beta_{0} \big) \, y'(r) + \big(\gamma_{1} \, r +\gamma_{0} \big) \, y(r) = 0 \, .
\end{equation}
As before, we want to extract a subclass of equations with a two-term recurrence relationship between successive coefficients in its series solution. We note by comparing equation \eqref{eq3.19} and equation \eqref{eq3.3} that $h+2-m = 3$ and $3>2-m\geq 0$, so $-1< m\leq 2$. Three cases $(m,h)=(0,1), (m,h)=(1,2)$ and $(m,h)=(2,3)$ follow from these constraints, and we consider each of these cases separately: 
\begin{itemize}
\item In the case where $m=0$ and $h=1$, equation \eqref{eq3.3} reads, with $q_{2,h=2} = \alpha_{3}$, $q_{1,h=2} = \beta_{2}$, $q_{0,h=2} = c_{1}$, $q_{2,0}=\alpha_{2}$, $q_{1,0}=\beta_{1}$, $q_{0,0}=\gamma_{0}$,
\begin{align}\label{eq3.20}
[\alpha_{2}r^{2}+\alpha_{3}\,r^{3}]\, u''(r)+[\beta_{1}r+\beta_{2}\,r^{2}]\, u'(r)+(\gamma_{0}+\gamma_{1}\,r)\, u(r)=0 \, ,\qquad \gamma_{0}\neq 0 \, ,
\end{align}
For this differential equation, the recurrence relation \eqref{eq3.5} implies
\begin{align}\label{eq3.21}
c_{k}=- \dfrac{\alpha_{3}\, (k+\lambda-1)(k+\lambda-2)+\beta_{2}\, (k+\lambda-1) +\gamma_{1}}{\alpha_{2}\, (k+\lambda)(k+\lambda-1)+\beta_{1}\, (k+\lambda) +\gamma_{0}}c_{k-1} \, , \qquad c_{0}\neq 0 \, , 
\end{align}
where $\lambda$ are the roots of the indicial equation $\alpha_{2}\lambda(\lambda-1)+\beta_{1}\lambda+\gamma_{0}=0$. The values of $\lambda$ with \eqref{eq3.21} gives the following two linearly independent series solutions for the differential equation \eqref{eq3.20}:
\begin{align}\label{eq3.22}
u_{1}(r)&=z^{\frac{\alpha_{2}-\beta_{1}+\sqrt{(\alpha_{2}-\beta_{1})^{2}-4 \alpha_{2}\gamma_{0}}}{2 \alpha_{2}}} {}_{2}F_{1}\left(\frac{\beta_{2}}{2\alpha_{3}}-\frac{\beta_{1}}{2\alpha_{2}}+\frac{\sqrt{(\alpha_{2}-\beta_{1})^{2}-4 \alpha_{2}\gamma_{0}}}{2 \alpha_{2}}-\frac{\sqrt{(\alpha_{3}-\beta_{2})^{2}-4 \alpha_{3}\gamma_{1}}}{2 \alpha_{3}} \,, \right.\notag\\
&\frac{\beta_{2}}{2\alpha_{3}}-\frac{\beta_{1}}{2\alpha_{2}}+\frac{\sqrt{(\alpha_{2}-\beta_{1})^{2}-4 \alpha_{2}\gamma_{0}}}{2 \alpha_{2}}+\frac{\sqrt{(\alpha_{3}-\beta_{2})^{2}-4 \alpha_{3}\gamma_{1}}}{2 \alpha_{3}};\left.\frac{\alpha_{2}+\sqrt{(\alpha_{2}-\beta_{1})^{2}-4\alpha_{2}\gamma_{0}}}{\alpha_{2}};-\frac{\alpha_{3} }{\alpha_{2}}\, r\right) \, ,
\end{align}
and
\begin{align}\label{eq3.23}
u_{2}(r)&=r^{\frac{\alpha_{2}-\beta_{1}-\sqrt{(\alpha_{2}-\beta_{1})^{2}-4 \alpha_{2}\gamma_{0}}}{2 \alpha_{2}}} {}_{2}F_{1}\left(\frac{\beta_{2}}{2\alpha_{3}}-\frac{\beta_{1}}{2\alpha_{2}}-\frac{\sqrt{(\alpha_{2}-\beta_{1})^{2}-4 a_{2}\gamma_{0}}}{2 \alpha_{2}}-\frac{\sqrt{(\alpha_{3}-\beta_{2})^{2}-4 \alpha_{3}\gamma_{1}}}{2 \alpha_{3}} \, ,\right.\notag\\
&\frac{b_{2}}{2a_{3}}-\frac{b_{1}}{2a_{2}}-\frac{\sqrt{(a_{2}-b_{1})^{2}-4 a_{2}\gamma_{0}}}{2 a_{2}}+\frac{\sqrt{(a_{3}-b_{2})^{2}-4 a_{3}\gamma_{1}}}{2 a_{3}};\left.\frac{a_{2}-\sqrt{(a_{2}-b_{1})^{2}-4a_{2}\gamma_{0}}}{a_{2}};-\frac{a_{3} }{a_{2}}\, r\right) \, .
\end{align}
\item In the case where $m=1$ and $h=2$, equation \eqref{eq3.3} reads, 
\begin{equation}\label{eq3.24}
[\alpha_{1}\, r+\alpha_{3} \,r^{3}]\, u''(r) +[\beta_{0}\, +\beta_{2}\, r^{2}]\, u'(r)+\gamma_{1}\, r\,  u(r)=0 \, ,\qquad \beta_{0}\neq 0 \, ,
\end{equation}
For this differential equation, the recurrence relation \eqref{eq3.5} implies
\begin{align}\label{eq3.25}
c_{k}=- \dfrac{\alpha_{3}\, (k+\lambda-2)(k+\lambda-3)+\beta_{2}\, (k+\lambda-2) +\gamma_{1}}{\alpha_{1}\, (k+\lambda)(k+\lambda-1)+\beta_{0}\, (k+\lambda)}c_{k-2} \, ,
\end{align}
where $\lambda$ are the roots of the indicial equation $a_{1}\, \lambda(\lambda-1)+b_{0}\, \lambda =0$. The values of $\lambda$ with \eqref{eq3.25} gives the following two linearly independent series solutions for the differential equation \eqref{eq3.24}:
\begin{align}\label{eq3.26}
u_{1}(r)&=\, _{2}F_{1}\left(\frac{\beta_{2}}{4\alpha_{3}}-\frac{1}{4}-\frac{\sqrt{(\alpha_{3}-\beta_{2})^{2}-4 \alpha_{3}\gamma_{1}}}{4 \alpha_{3}} \, ,\frac{\beta_{2}}{4\alpha_{3}}-\frac{1}{4}+\frac{\sqrt{(\alpha_{3}-\beta_{2})^{2}-4 \alpha_{3}\gamma_{1}}}{4 \alpha_{3}}; \frac{\alpha_{1}+\beta_{0}}{2 \alpha_{1}};-\frac{\alpha_{3} r^{2}}{\alpha_{1}}\right) \, ,
\end{align}
and
\begin{align}\label{eq3.27}
u&_{2}(r)=z^{1-\frac{\beta_{0}}{\alpha_{1}}}\, \times\\
& _{2}F_{1}\left(\frac14+\frac{\beta_{2}}{4\alpha_{3}}-\frac{\beta_{0}}{2\alpha_{1}}-\frac{\sqrt{(\alpha_{3}-\beta_{2})^{2}-4 \alpha_{3}\gamma_{1}}}{4\alpha_{3}} \, ,\frac14+\frac{\beta_{2}}{4\alpha_{3}}-\frac{\beta_{0}}{2\alpha_{1}}+\frac{\sqrt{(\alpha_{3}-\beta_{2})^{2}-4 \alpha_{3}\gamma_{1}}}{4\alpha_{3}};\frac{3}{2}-\frac{\beta_{0}}{2\alpha_{1}};-\frac{\alpha_{3} r^{2}}{\alpha_{1}}\right) \, .
\end{align}
\item In the case where $m=2$ and $h=3$, equation \eqref{eq3.3} reads, 
\begin{equation}\label{eq3.28}
(\alpha_{0}+\alpha_{3} \,r^{3})\, u''(r)+\beta_{2}\, r^{2}\, u'(r)+\gamma_{1} r\,  u(r)=0 \, ,\qquad \alpha_{0}\neq 0 \, ,
\end{equation}
For this differential equation, the recurrence relation \eqref{eq3.5} implies
\begin{align}\label{eq3.29}
c_{k}=- \dfrac{\alpha_{3}\, (k+\lambda-3)(k+\lambda-4)+\beta_{2}\, (k+\lambda-3) +\gamma_{1}}{\alpha_{0}\, (k+\lambda)(k+\lambda-1)} c_{k-3} \, , \qquad c_{0}\neq 0 \, , 
\end{align}
where $\lambda$ are the roots of the indicial equation $\alpha_{0}\,\lambda\,(\lambda-1)=0$. The values of $\lambda$ with \eqref{eq3.29} gives the following two linearly independent series solutions for the differential equation \eqref{eq3.28}:
\begin{align}\label{eq3.30}
u_{1}(r)&=\, _{2}F_{1}\left(-\frac{\sqrt{(\alpha_{3}-\beta_{2})^{2}-4\alpha_{3}\gamma_{1}}+\alpha_{3}-\beta_{2}}{6\alpha_{3}} \, ,\frac{\sqrt{(\alpha_{3}-\beta_{2})^{2}-4\alpha_{3}\gamma_{1}}-\alpha_{3}+\beta_{2}}{6\alpha_{3}};\frac{2}{3};-\frac{\alpha_{3}}{\alpha_{0}}\, r^{3}\right)G  \, ,
\end{align}
and
\begin{align}\label{eq3.31}
u_{2}(r)=r\, _{2}F_{1}\left(\frac{-\sqrt{(\alpha_{3}-\beta_{2})^{2}-4 \alpha_{3}\gamma_{1}}+\alpha_{3}+\beta_{2}}{6\alpha_{3}} \, ,\frac{\sqrt{(\alpha_{3}-\beta_{2})^{2}-4 \alpha_{3}\gamma_{1}}+\alpha_{3}+\beta_{2}}{6\alpha_{3}};\frac{4}{3};-\frac{\alpha_{3}}{\alpha_{0}}\, r^{3}\right)G  \, .
\end{align}
\end{itemize}

\begin{remark} For the polynomial coefficients differential equation \eqref{eq3.1}, there are $n$-generic equations with the series solutions explicitly found using the two-term recurrence relation \eqref{eq3.5}. The closed-form solutions of these equations are given using \eqref{eq3.7}. Indeed, for $h+2-m=n, n>2-m\geq 0$, we have $n$ cases given $h=n+m-2$, $2-n< m\leq 2$, $n\geq 2$. Some of these equations are as follows (See Theorem \ref{Thm3.1}):
\begin{align} \label{eq3.32}
\text{$n$-equations}~ \left\{ \begin{array}{ll}
      (\alpha_{0}+\alpha_{n} r^{n}) \, y''+\beta_{n-1}z^{n-1}\, y'+\varepsilon_{n-2}r^{n-2}\, y=0\,, & \alpha_{0}\neq 0\\
(\alpha_{1} r +\alpha_{n} r^{n})\, y''+(\beta_{0}+\beta_{n-1}r^{n-1})\, y'+\varepsilon_{n-2}r^{n-2}\, y=0\,, & \beta_{0}\neq 0\\
(\alpha_{2} r^{2} +\alpha_{n} r^{n})\, y''+(\beta_{1} r+\beta_{n-1}r^{n-1})\, y'+(\varepsilon_{0}+\varepsilon_{n-2}r^{n-2})\, y=0\,, & \varepsilon_{0}\neq 0\\
(\alpha_{3} r^{3} +\alpha_{n} r^{n})\, y''+(\beta_{2} r^{2}+\beta_{n-1}r^{n-1})\, y'+(\varepsilon_{1} r +\varepsilon_{n-2}r^{n-2})\, y=0\,,\\
(\alpha_{4} r^{4} +\alpha_{n} r^{n})\, y''+(\beta_{3} r^{3}+\beta_{n-1}r^{n-1})\, y'+(\varepsilon_{2} r^{2} +\varepsilon_{n-2}r^{n-2})\, y=0\,,\\
\vdots\\
(\alpha_{n-1} r^{n-1} +\alpha_{n} r^{n})\, y''+(\beta_{n-2} r^{n-2}+\beta_{n-1}r^{n-1})\, y'+(\varepsilon_{n-3} r^{n-3} +\varepsilon_{n-2}r^{n-2})\, y=0\,.\\
\end{array} \right. 
\end{align} 
\end{remark}
\section{Theorems And Algorithm}\label{Sec4}

In this section we present three elementary theorems that classify the polynomial solutions of the general differential equation \eqref{eq1.3}. Following that, we will give a brief description of the Mathematica\textsuperscript{\textregistered} program that was developed to accompany the present work. A link to the complete program, available for direct uses, is also provided for direct applications of these theorems. 

\subsection{Theorems}

Theorem \ref{Thm4.1} applies to $2^{n}$ equations of \eqref{eq1.3} and gives a recurrence relation needed to compute polynomial solutions in the neighbourhood of an ordinary point. Theorems \ref{Thm4.2} and \ref{Thm4.3} give the polynomial solutions for the  $2^{n - 1} + 2^{n - 2}$ equations from \eqref{eq1.3} about a regular singular point. The proofs of Theorems \ref{Thm4.1} and \ref{Thm4.2} may be found in the Appendix, where the proof of Theorem \ref{Thm4.3} follows closely to the proof of Theorem \ref{Thm4.2} and is therefore omitted.\\

\begin{theorem}\label{Thm4.1} $\boldsymbol{(\alpha_{0} \neq 0)}$\\
For $\alpha_{0} \neq 0$, the second-order linear differential equation 
\begin{equation}\label{eq4.1}
\left(\sum_{k=0}^{n} \alpha_{k} \, r^{k} \right) \, y''(r) +\left( \sum_{k=0}^{n-1} \beta_{k} \, r^{k} \right) \, y'(r) - \left(\sum_{k=0}^{n-2} \tau_{k} \, r^{k} \right) \, y(r) = 0 \, ,
\qquad n \geq 2 \, ,
\end{equation}
admits the solution $y(r)=\mathcal{P}_{m}(r)=\sum_{i=0}^{m} \mathcal{C}_{i} \, r^{i}$ in the neighbourhood of the ordinary point $r=0$ and valid to the nearest nonzero singular point of $\mathcal{P}_{m}(r)=0$. Here $\mathcal{P}_{m}(r)$ is an $m^{th}$-degree polynomial if
\begin{equation}\label{eq4.2}
\sum_{j=0}^{n} \Big[ \alpha_{n-j} \, (\ell-n+j+2) \, (\ell-n+j+1) + \beta_{n-j-1} \, (\ell-n+j+2) - \tau_{n-j-2} \Big] \, \mathcal{C}_{\ell-n+j+2} = 0 \, ,
\end{equation} 
for all $\ell = 0, \, 1, \, 2, \, \ldots, \, m+n-2$, and $\tau_{i} = \beta_{i} = 0$ if $i<0$. The equation corresponding to the value of $\ell = m+n-2$ gives the necessary condition \eqref{eq2.1} for the existence of the $m^{th}$-degree polynomial solution 
\begin{equation}\label{eq4.3}
\tau_{n-2} = \alpha_{n} \, m \, (m-1) + \beta_{n-1} \, m \, ,\qquad m = 0, \, 1, \, 2, \, \ldots \, .
\end{equation}
The remaining $(m+n-2)$ linear equations consist of $(n-2)$ sufficient conditions that relate $\{ \alpha_{k} \}_{k=0}^{n}$ and $\{ \beta_{k} \}_{k=0}^{n-1}$ with $\{ \tau_{k} \}_{k=0}^{n-2}$, in addition to the $m$ linear equations required to evaluate the coefficients $\{ \mathcal{C}_{i} \}_{i=1}^{m}$ of the polynomial solution in terms of $\mathcal{C}_{0} \neq 0$.
\end{theorem}

\begin{theorem}\label{Thm4.2} $\boldsymbol{(\alpha_{0}=0, \, \alpha_{1}\neq 0)}$\\
For $\alpha_{0}=0$ and $\alpha_{1}\neq 0$, the second-order linear differential equation 
\begin{equation}\label{eq4.4}
\left( \sum_{k=1}^{n} \alpha_{k} \, r^{k} \right) y''(r) + \left( \sum_{k=0}^{n-1} \beta_{k} \, r^{k} \right) y'(r) - \left(\sum_{k=0}^{n-2} \tau_{k} \, r^{k}\right) y(r) = 0 \, , \qquad n\geq 2 \, ,
\end{equation}
admits the solution $y(r)=r^{s} \, \mathcal{P}_{m}(r) = r^{s} \, \sum_{i=0}^{m} \mathcal{C}_{i} \, r^{i}$ in the neighbourhood of the regular singular point $r=0$, where $\mathcal{P}_{m}(r)$ is an $m^{th}$-degree polynomial if
\begin{equation}\label{eq4.5}
\sum_{j=0}^{n-1} \Big[ \alpha_{n-j} \, (\ell-n+j+s+2) \, (\ell-n+j+s+1) + \beta_{n-j-1} \, (\ell-n+j+s+2) - \tau_{n-j-2} \Big] \, \mathcal{C}_{\ell-n+j+2} = 0 \, ,
\end{equation} 
for all $\ell = 0, \, 1, \, 2, \, \ldots, \, m+n-2$, and $\tau_{i} = 0$ if $i<0$.  Here $s$ is a root of the indicial equation: $s(s-1) + ( \beta_{0}/\alpha_{1} )s = 0$, so $s \in \big\{ 0 \, , \, 1-(\beta_{0}/\alpha_{1}) \big\}$.
\end{theorem}

\begin{theorem}\label{Thm4.3} $\boldsymbol{( \alpha_{0} = \alpha_{1} = \beta_{0} = 0$, $\alpha_{2} \neq 0)}$\\
For $\alpha_{0} = \alpha_{1} = \beta_{0} = 0$ and $\alpha_{2} \neq 0$, the second-order linear differential equation
\begin{align}\label{eq4.6}
\left( \sum_{k=2}^{n} \alpha_{k} \, r^{k} \right) y''(r) + \left( \sum_{k=1}^{n-1} \beta_{k} \, r^{k} \right) y'(r) - \left( \sum_{k=0}^{n-2} \tau_{k} \, r^{k} \right) y(r) = 0 \, , \qquad n \geq 2 \, ,
\end{align}
admits the solution $y(r)=r^{s} \, \mathcal{P}_{m}(r) = r^{s} \, \sum_{i=0}^{m} \mathcal{C}_{i} \, r^{i}$ in the neighbourhood of the regular singular point $r=0$,  where $\mathcal{P}_{m}(r)$ is an $m^{th}$-degree polynomial if
\begin{equation}\label{eq4.7}
\sum_{j=0}^{n-2} \Big[ \alpha_{n-j} \, (\ell-n+j+s+2) \, (\ell-n+j+s+1) + \beta_{n-j-1} \, (\ell-n+j+s+2) - \tau_{n-j-2} \Big] \, \mathcal{C}_{\ell-n+j+2} = 0 \, ,
\end{equation} 
for all $\ell = 0, \, 1, \, 2, \, \ldots, \, m+n-2$.  Here $s$ is a non-zero root of the indicial equation: $s(s-1) + ( \beta_{0}/\alpha_{1})s - \tau_{0}\alpha_{2}= 0$, so $s \in \left\{ \Big( \alpha_{2} - \beta_{1} \pm \sqrt{(\alpha_{2} - \beta_{1})^{2} + 4 \, \alpha_{2} \, \tau_{0}} \, \Big) / \big( 2 \, \alpha_{2} \big) \right\}$.
\end{theorem}

\begin{remark} (Cauchy-Euler equation): In the case of $n=2$, equation \eqref{eq4.6} reduces to the equation, 
\begin{equation} \label{eq4.8}
\alpha_{2} \, r^{2} \, y''(r) + \beta_{1} \, r \, y'(r) - \tau_{0} \, y(r) = 0 \, , 
\end{equation} 
for which the recurrence relation \eqref{eq4.7} reduces to: 
\begin{align}\label{eq4.9}
\Big[\alpha_{2} \, (\ell+s)(\ell+s-1) + \beta_{1} \, (\ell+s) - \tau_{0}\Big]\mathcal{C}_{\ell}=0 \, ,\quad \ell = 0,\,1,\,\ldots,\,m \, .
\end{align}
Clearly when $\ell=m$ and if $\mathcal{C}_{m} \neq 0$ then equation \eqref{eq4.9} gives the necessary condition: 
\begin{equation}\label{eq4.10}
\alpha_{2} \, (m + s) (m + s - 1) + \beta_{1} \, (m + s) - \tau_{0} = 0 \, , \qquad \mathcal{C}_{m} \neq 0 \,,
\end{equation}
as expected for the solutions of the Cauchy-Euler equation.
\end{remark}

\subsection{The Mathematica\textsuperscript{\textregistered} Program}
A Mathematica\textsuperscript{\textregistered} algorithm and accompanying GUI (Figure \ref{Fig1}) were developed to assist with this research project. The algorithm uses Cramer's rule to compute the coefficients of the polynomial solutions $\{\mathcal{C}_{i}\}_{i=1}^{m}$ for the equation
\begin{equation}\label{eq4.11}
\left(\sum_{k=0}^{n} \alpha_{k} \, r^{k} \right) \, y''(r) +\left( \sum_{k=0}^{n-1} \beta_{k} \, r^{k} \right) \, y'(r) - \left(\sum_{k=0}^{n-2} \tau_{k} \, r^{k} \right) \, y(r) = 0 \, ,
\qquad n \geq 2 \, .
\end{equation}
The main functions of the program are: 
\begin{itemize}
\item \textbf{Theorem1Fn}: Evaluates the coefficients $\{\mathcal{C}_{i}\}_{i=0}^{m}$ of the $m^{th}$-degree polynomial solution in the case where the equation is in the form of Theorem \ref{Thm4.1}.           
\item \textbf{Theorem2Fn}: Evaluates the coefficients $\{\mathcal{C}_{i}\}_{i=0}^{m}$ of the $m^{th}$-degree polynomial solution in the case where the equation is in the form of Theorem \ref{Thm4.2}.
\item \textbf{Theorem3Fn}: Evaluates the coefficients $\{\mathcal{C}_{i}\}_{i=0}^{m}$ of the $m^{th}$-degree polynomial solution in the case where the equation is in the form of Theorem \ref{Thm4.3}. 
\item \textbf{GeneralRecursionFn}: For given values of $m$ and $n$, this function generates an $(m+n-1)\times(m+1)$ matrix containing the equations needed to evaluate the polynomial coefficients, in addition to the sufficient and necessary conditions. This function serves as a general recurrence relation that is called by Theorem1Fn, Theorem2Fn, and Theorem3Fn. 
\end{itemize}
A complete Mathematica\textsuperscript{\textregistered} program may be obtained by contacting the authors, or via the following Github repository: \href{https://github.com/KyleBryenton/OPSOLDE}{https://github.com/KyleBryenton/OPSOLDE}. 
\begin{center}
\begin{figure}[h!]
\includegraphics[width=\textwidth]{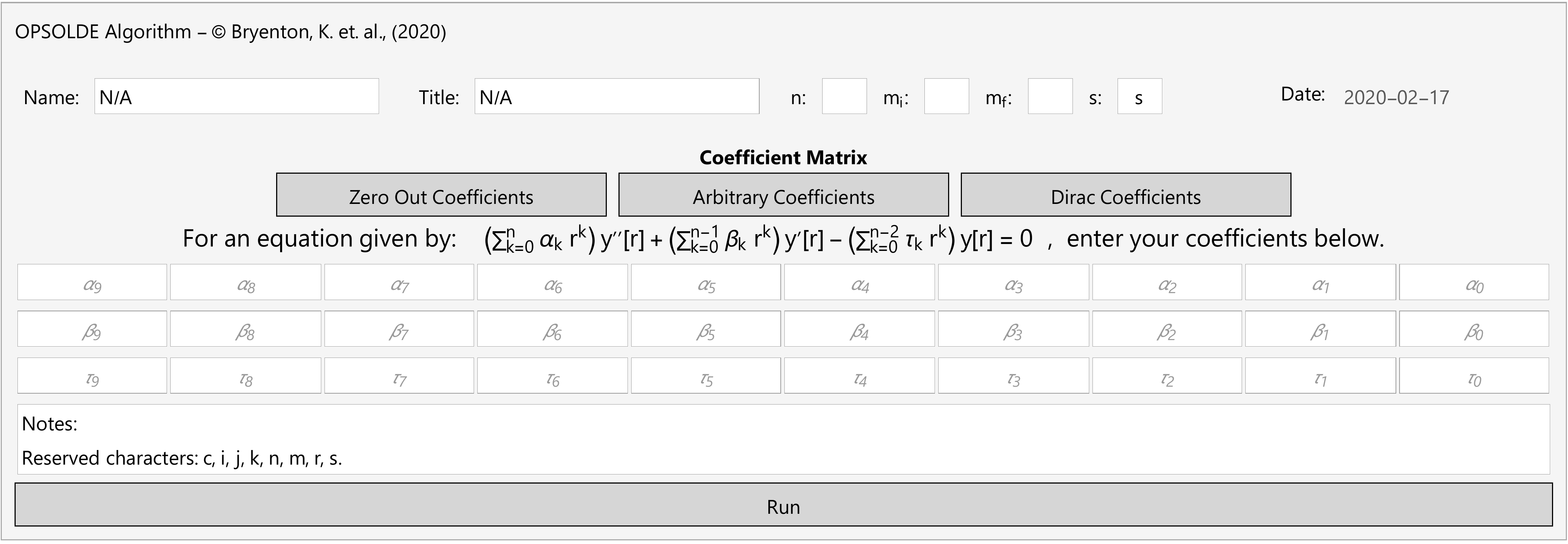}
\caption{The graphical user interface of the OPSOLDE Algorithm. Upon entering a value for $2\leq n\leq 9$, the coefficient field will resize appropriately and permit the user to choose the values for each coefficient.}\label{Fig1}
\label{GUI}
\end{figure}
\end{center}

\section{Applications} \label{Sec5}
The simplicity of the above theorems lay in their constructive approach to explicitly computing the polynomial solution coefficients and to provide all the necessary and sufficient conditions for the existence of these polynomials. Further, they can be easily implemented using any available symbolic algebra system. In this section, we show how to make use of these theorems.

\subsection{General Case ($n=4$)}
To illustrate the constructive approach of Theorem \ref{Thm4.1}, we consider equation \eqref{eq4.1} in the case of $n=4$ which reads:
\begin{align}\label{eq5.1}
( \alpha_{4}\,r^{4} + \alpha_{3}\,r^{3} + \alpha_{2}\,r^{2} + \alpha_{1}\,r + \alpha_{0} ) \, y''(r)  +&  ( \beta_{3}\,r^{3} + \beta_{2}\,r^{2} + \beta_{1} r + \beta_{0} ) \, y'(r)-  ( \tau_{2}\,r^{2} + \tau_{1}\,r + \tau_{0} ) \, y(r) = 0 \, , \notag\\
& \qquad \qquad \qquad \text{where} \quad \alpha_{0} \neq 0 \, , \quad \alpha_{k} \, , \, \beta_{k} \, , \, \tau_{k} \in \mathbb{R} \quad \forall\, k \, .
\end{align}
For the $m^{th}$-degree polynomial solutions $\mathcal{P}_{m}(r)\equiv y(r)=\sum_{i=0}^{m} \mathcal{C}_{i} \, r^{i}$, \eqref{eq4.2} gives the five term recurrence-relation:  
\begin{align}\label{eq5.2}
[\alpha_{4} \, (\ell-2) \, (\ell-3) &+ \beta_{3} \, (\ell-2) - \tau_{2} ] \mathcal{C}_{\ell-2}+[ \alpha_{3} \, (\ell-1) \, (\ell-2) + \beta_{2} \, (\ell-1) - \tau_{1}] \mathcal{C}_{\ell-1}\notag\\
&+ [ \alpha_{2} \, \ell \, (\ell-1) + \beta_{1} \, \ell - \tau_{0} ] \mathcal{C}_{\ell} +[ \alpha_{1} \, (\ell+1) \, \ell + \beta_{0} \, (\ell+1)] \mathcal{C}_{\ell+1} + [ \alpha_{0} \, (\ell+2) \, (\ell+1) ] \mathcal{C}_{\ell+2} = 0 \, ,
\end{align}
for $\ell = 0,\,1,\,2,\,\ldots,\,m+2$ and $C_{j}=0$ for  $j<0$.\\

For the zero-degree ($m=0$) polynomial solution, $\mathcal{P}_{0}(r)=1$, $\ell=0,\,1,\,2$. Noting that $\mathcal{C}_{j}=0$, for all $j \neq 0$, it follows for $\ell=0$ and $\ell=1$ that $\tau_{0}=0$ and $\tau_{1}=0$ respectively, while for $\ell=2$, the necessary condition $\tau_{2}=0$ follows.\\

For the first-degree ($m=1$) polynomial solution, $\mathcal{P}_{1}(r)=\mathcal{C}_{0}+\mathcal{C}_{1}\,r$, $\ell=0,\,1,\,2,\,3$. Noting that $\mathcal{C}_{j}=0$, for all $j \neq 0,\,1$, we have for $\ell=0$ that $-\tau_{0} \, \mathcal{C}_{0} + \beta_{0} \, \mathcal{C}_{1}=0$, while for $\ell=1$ and $\ell=2$ we have $-\tau_{1} \, \mathcal{C}_{0} + ( \beta_{1} - \tau_{0} ) \, \mathcal{C}_{1}=0$ and $-\tau_{2} \, \mathcal{C}_{0} + ( \beta_{2} - \tau_{1} ) \, \mathcal{C}_{1}=0$ respectively. For $\ell=3$ it follows $(\beta_{3}-\tau_{2}) \, \mathcal{C}_{1}=0$ from which it is necessary that $\tau_{2}=\beta_{3}$, while the first equation ($\ell=0$) gives $\mathcal{C}_{1}=\tau_{0} \, \mathcal{C}_{0}/\beta_{0}$. Using this value for $\mathcal{C}_{1}$ and letting $\mathcal{C}_{0}=1$, the second and third equations ($\ell=1,2$)  give $-\tau_{1}\,\beta_{0}+(\beta_{1}-\tau_{0})\, \tau_{0}=0$ and $-\tau_{2}\,\beta_{0}+(\beta_{2}-\tau_{1})\, \tau_{0}=0$ respectively. We can write these conditions compactly as:
\begin{align}
& \ell=0: \quad \mathcal{C}_{1}=\tau_{0}/\beta_{0} \, , \qquad  \quad \ell=1,\,2: \left\{ \begin{array}{l}
\left| \begin{array}{ll}
-\tau_{0} & \beta_{0} \\
-\tau_{1} & \beta_{1}-\tau_{0}
\end{array} \right| = 0 \, , \quad \mbox{(S1)} \qquad 
\left| \begin{array}{ll}
-\tau_{0} & \beta_{0} \\
-\tau_{2} & \beta_{2}-\tau_{1}
\end{array} \right| = 0 \, , \quad \mbox{(S2)}
\end{array} \right.  \notag\\
& \ell=3: \quad \tau_{2}=\beta_{3} \, , \quad \mbox{(NC)} \label{eq5.3}
\end{align}
where (S1) and (S2) refer to the two sufficient conditions and (NC) refers to the necessary condition.\\

For the second-degree $(m=2)$ polynomial solution $\mathcal{P}_{2}(r)= 1+\mathcal{C}_{1}\,r + \mathcal{C}_{2}\,r^{2}$, noting that $\mathcal{C}_{j}=0$, for all $j \neq 0,\,1,\,2,$  where $\ell=0,\,1,\,2,\,3,\,4$ now yields
\begin{align}
& \ell=0,\,1: \quad \left\{ 
\begin{array}{ll}
\mathcal{C}_{1} &= \frac{
\left|\begin{array}{ll} 
\tau_{0} & 2 \, \alpha_{0} \\  
\tau_{1}& 2 \left( \alpha_{1} + \beta_{0} \right)
\end{array}\right|}{\left|\begin{array}{ll} 
\beta_{0} & 2 \, \alpha_{0} \\ 
\beta_{1} - \tau_{0} & 2 \left( \alpha_{1} + \beta_{0} \right)
\end{array}\right|} \, , \qquad
\mathcal{C}_{2} = \frac{
\left|\begin{array}{ll} 
\beta_{0} & \tau_{0} \\
\beta_{1} - \tau_{0} & \tau_{1}
\end{array}\right|}{\left|\begin{array}{ll} 
\beta_{0} & 2 \, \alpha_{0} \\
\beta_{1} - \tau_{0} & 2 \left( \alpha_{1} + \beta_{0} \right)
\end{array}\right|} \, ,
\end{array} \right.\notag\\
& \ell=2,3: \quad \left\{ \begin{array}{ll}
\left|\begin{array}{lll} 
-\tau_{0} & \beta_{0} & 2 \, \alpha_{0} \\
-\tau_{1} & \beta_{1} - \tau_{0} & 2 \left( \alpha_{1} + \beta_{0} \right) \\
-\tau_{2} & \beta_{2} - \tau_{1} & 2 \left( \alpha_{2} + \beta_{1} \right) - \tau_{0}
\end{array}\right|=0 \, , \quad \mbox{(S1)} \qquad \left|\begin{array}{lll} 
-\tau _{0} & \beta_{0} & 2 \, \alpha_{0} \\
-\tau_{1} & \beta_{1}-\tau _{0} & 2 \left(\alpha_{1}+\beta_{0}\right) \\
0 & \beta_{3}-\tau_{2} & 2 \left(\alpha_{3}+\beta_{2}\right)-\tau_{1} \\
\end{array}\right|=0 \, , \quad \mbox{(S2)} \\
\end{array} \right. \notag\\ \notag\\
& \ell=4: \quad  2 \, \alpha_{4} + 2 \, \beta_{3} - \tau_{2} = 0 \, .  \quad \mbox{(NC)} \label{eq5.4}
\end{align}

For the third degree $(m=3)$ polynomial solution $\mathcal{P}_{3}(r)=1+\mathcal{C}_{1}\,r+\mathcal{C}_{2}\,r^{2}+\mathcal{C}_{3}\,r^{3}$, noting that $\mathcal{C}_{j}=0$, for all $j\neq 0,\,1,\,2,\,3$, we have for $\ell=0,\,1,\,2,\,3,\,4,\,5$
 
\begin{align*}
& \ell=0,\,1,\,2: \quad \left\{ 
\begin{array}{ll}
\mathcal{C}_{1} &= \frac{
\left|\begin{array}{lll} 
\tau _{0} & 2 \, \alpha_{0} & 0 \\
\tau_{1} & 2 \left(\alpha_{1}+\beta_{0}\right) & 6 \, \alpha_{0} \\
\tau_{2} & 2 \left(\alpha_{2}+\beta_{1}\right)-\tau _{0} & 3 \left(2 \alpha_{1}+\beta_{0}\right) \\
\end{array}\right|}{\left|\begin{array}{lll} 
\beta_{0} & 2 \, \alpha_{0} & 0 \\
\beta_{1}-\tau _{0} & 2 \left(\alpha_{1}+\beta_{0}\right) & 6 \, \alpha_{0} \\
\beta_{2}-\tau_{1} & 2 \left(\alpha_{2}+\beta_{1}\right)-\tau _{0} & 3 \left(2 \alpha_{1}+\beta_{0}\right) \\
\end{array}\right|} \, , \\
\\
\mathcal{C}_{2} &= \frac{
\left|\begin{array}{lll} 
\beta_{0} & \tau _{0} & 0 \\
\beta_{1}-\tau _{0} & \tau_{1} & 6 \, \alpha_{0} \\
\beta_{2}-\tau_{1} & \tau_{2} & 3 \left(2 \alpha_{1}+\beta_{0}\right) \\
\end{array}\right|}{\left|\begin{array}{lll} 
\beta_{0} & 2 \, \alpha_{0} & 0 \\
\beta_{1}-\tau _{0} & 2 \left(\alpha_{1}+\beta_{0}\right) & 6 \, \alpha_{0} \\
\beta_{2}-\tau_{1} & 2 \left(\alpha_{2}+\beta_{1}\right)-\tau _{0} & 3 \left(2 \alpha_{1}+\beta_{0}\right) \\
\end{array}\right|} \, , \\
\\
\mathcal{C}_{3} &= \frac{
\left|\begin{array}{lll} 
\beta_{0} & 2 \, \alpha_{0} & \tau _{0} \\
\beta_{1}-\tau _{0} & 2 \left(\alpha_{1}+\beta_{0}\right) & \tau_{1} \\
\beta_{2}-\tau_{1} & 2 \left(\alpha_{2}+\beta_{1}\right)-\tau _{0} & \tau_{2} \\
\end{array}\right|}{\left|\begin{array}{lll} 
\beta_{0} & 2 \, \alpha_{0} & 0 \\
\beta_{1}-\tau _{0} & 2 \left(\alpha_{1}+\beta_{0}\right) & 6 \, \alpha_{0} \\
\beta_{2}-\tau_{1} & 2 \left(\alpha_{2}+\beta_{1}\right)-\tau _{0} & 3 \left(2 \alpha_{1}+\beta_{0}\right) \\
\end{array}\right|} \, ,
\end{array} \right.  
\end{align*}
\begin{align}
& \ell=3,\,4: \quad \left\{ \begin{array}{ll}
\left|\begin{array}{llll} 
-\tau _{0} & \beta_{0} & 2 \, \alpha_{0} & 0 \\
-\tau_{1} & \beta_{1}-\tau _{0} & 2 \left(\alpha_{1}+\beta_{0}\right) & 6 \, \alpha_{0} \\
-\tau_{2} & \beta_{2}-\tau_{1} & 2 \left(\alpha_{2}+\beta_{1}\right)-\tau _{0} & 3 \left(2 \, \alpha_{1}+\beta_{0}\right) \\
0 & \beta_{3}-\tau_{2} & 2 \left(\alpha_{3}+\beta_{2}\right)-\tau_{1} & 6 \, \alpha_{2}+3 \, \beta_{1}-\tau _{0} \\
\end{array}\right|=0 \, , \quad \mbox{(S1)} \\
\\
\left|\begin{array}{llll} 
-\tau _{0} & \beta_{0} & 2 \, \alpha_{0} & 0 \\
-\tau_{1} & \beta_{1}-\tau_{0} & 2 \left(\alpha_{1}+\beta_{0}\right) & 6 \, \alpha_{0} \\
-\tau_{2} & \beta_{2}-\tau_{1} & 2 \left(\alpha_{2}+\beta_{1}\right)-\tau_{0} & 3 \left(2 \, \alpha_{1}+\beta_{0}\right) \\
0 & 0 & 2 \left(\alpha_{4}+\beta_{3}\right)-\tau_{2} & 6 \, \alpha_{3}+3 \, \beta_{2}-\tau_{1} \\
\end{array}\right|=0 \, , \quad \mbox{(S2)}\notag\\
\end{array} \right.\\
& \ell=5: \quad 6 \, \alpha_{4} + 3 \, \beta_{3} - \tau_{2}\ = 0 \, . \quad \mbox{(NC)} \label{eq5.5}
\end{align}
Higher-order polynomial solutions may be obtained similarly in this constructive manner.

\subsection{Two-Dimensional Dirac Equation}
The general covariant $(2+1)$-dimensional Dirac equation reads
\begin{align}\label{eq5.6}
i\,\gamma^{\mu}(\partial_{\mu}+i\,e\,A_{\mu})\,\Psi(t,\vec{r}) = \mathcal{M}\, \Psi(t,\vec{r}) \, ,
\end{align}
where $\mathcal{M}$ is the rest mass of the particle, $A_{\mu}=(A_{0} \, ,\,\vec{A})$ is the external electromagnetic field potential, and $-e$ is the charge of the particle. Here, $\gamma^{0} = \sigma_{z}$, $\gamma^{1} = i\,\sigma_{x}$, $\gamma^{2} = i\,\sigma_{y}$, where $\sigma_{x}$, $\sigma_{y}$, and $\sigma_{z}$ are the Pauli spin matrices. Using the generalized momentum operator $P_{k} = -i\,\partial_{k} + e\,A_{k}$ for $k = 1,\,2$ and the conventional summation over $\mu$, equation \eqref{eq5.6} may be written in the form
\begin{align}\label{eq5.7}
(i\,\partial_{t} - \sigma_{x}\,P_{y} + \sigma_{y}\,P_{x} - \sigma_{z}\,m -e\,A_{0})\,\Psi(t,\vec{r}) = 0 \, .
\end{align}
Chiang et al. \cite{ch2002} considered an electron moving in a Coulombic field, $A_{0} ={Z\,e}/{r}$, with a constant homogeneous magnetic field described by $A_{1} = -{B\,y}/{2}$ and $A_{2} ={B\,x}/{2}$. With this choice of potential and for $r>0$, \eqref{eq5.7} becomes
\begin{align}\label{eq5.8}
\bigg(i\,\partial_{t} - \bigg[\begin{array}{ll}
\mathcal{M} & -e^{-i \,\phi}(\partial_{r} - \partial_{\phi} + \frac{e\,B\,r}{2}) \\ 
e^{i \,\phi}(\partial_{r} + \partial_{\phi} - \frac{e\,B\,r}{2}) & -\mathcal{M}
\end{array}\bigg] -\frac{Z\,e^{2}}{r}\bigg) \, \Psi(t,\vec{r}) = 0 \, ,
\end{align}
where polar coordinates, the identities $P_{x}\pm i\,P_{y} = -i\,e^{\pm i \,\phi}(\partial_{r} \pm (\partial_{\phi} - ({e\,B\,r}/2)))$, $\partial_{x} = \cos\phi \, \partial_{r} - [\sin(\phi)/r]\, \partial_{\phi}$, and $\partial_{y} = \sin(\phi)\, \partial_{r} + [\cos(\phi)/r]\, \partial_{\phi}$ were employed \cite{ch2002}. Using separation of variables
\begin{align}\label{eq5.9}
\Psi(t,\vec{r}) = \frac{1}{\sqrt{r}}\,e^{-i\,E\,t}\,
\left(\begin{array}{ll}
F(r)\,e^{i\,l \,\phi} \\ 
G(r)\,e^{i\,(l+1)\,\phi}
\end{array} \right)\, ,
\end{align}
where $l \in \mathbb{Z}$, the Dirac radial equation \eqref{eq5.8} reads
\begin{align}
\frac{d\,F(r)}{dr} - \bigg(\frac{l+\frac{1}{2}}{r}+\frac{e\,B}{2}\,r\bigg) \, F(r)+\bigg(E+\mathcal{M}+\frac{Z\,e^{2}}{r}\bigg) \,G(r) &= 0 \, , \label{eq5.10}\\ 
\frac{d\,G(r)}{dr} + \bigg(\frac{l+\frac{1}{2}}{r}+\frac{e\,B}{2}\,r\bigg) \, G(r)-\bigg(E-\mathcal{M}+\frac{Z\,e^{2}}{r}\bigg) \, F(r) &= 0 \, ,  \label{eq5.11}
\end{align}
For a strong magnetic field, using the asymptotic solutions near $r=0$ and $r=\infty$, it is beneficial to assume an ansatz \cite{ch2002}
\begin{align}\label{eq5.12}
F(r) & = r^{\gamma}e^{-e\,B\,r^{2}/4}Q(r) \, , \quad G(r)  = r^{\gamma}e^{-e\,B\,r^{2}/4}P(r) \, ,
\end{align}
where $\gamma = \sqrt{(l+1/2)^{2}-(Z\,e^{2})^{2}}$. Substituting \eqref{eq5.12} into (\ref{eq5.10}-\ref{eq5.11}) yields
\begin{align}
&\frac{dQ(r)}{dr}+\frac{(\gamma-l-\frac{1}{2}-e\,B\,r^{2})}{r}Q(r)+\left(E+\mathcal{M}+\frac{Z\,e^{2}}{r}\right)P(r) = 0 \, , \label{eq5.13}\\
&\frac{dP(r)}{dr}+\frac{(\gamma+l+\frac{1}{2})}{r}P(r)-\left(E-\mathcal{M}+\frac{Z\,e^{2}}{r}\right)Q(r) = 0 \, . \label{eq5.14}
\end{align}
Solving \eqref{eq5.13} for $P(r)$ and substituting the resulting equation into \eqref{eq5.14} forms a second-order differential equation for $Q(r)$
\begin{align}\label{eq5.15}
& r(r+r_{0})\, Q''(r)-(e\,B\,r^{3}+e\,B\,r_{0}\,r^{2}-\big(\beta-1\big)\,r-\beta \,r_{0})\, Q'(r)+((e\,B+\varepsilon)r^{2}+(\varepsilon\, r_{0}+a)\,r+c-a\,r_{0})\, Q(r)=0 \, , \notag \\
& r_{0}\equiv Z\,e^{2}/(E+\mathcal{M}) \, , \quad \beta\equiv 2\,\gamma + 1, 
\quad a\equiv 2\,E\,Z\,e^{2} \, ,\quad\varepsilon\equiv E^{2}-\mathcal{M}^{2}-e\,B\,(l+\frac{5}{2}+\gamma), \quad c\equiv \gamma - l - \frac{1}{2} \, .
\end{align}
Equation \eqref{eq5.15} is an example of differential equations scattered over a vast range of applications in applied mathematics and theoretical physics \cite{ch2014,ch2015,ch2016,ch2002}. The coefficients $\{{\mathcal{C}_{i}}\}_{i=0}^{m}$ of the polynomial solutions $Q_{m}(r)=\sum_{i=0}^{m} \mathcal{C}_{i} \,r^{i}$ satisfy the following recurrence relation \eqref{eq4.5}:
\begin{align}\label{eq5.16}
&\left(E^{2}-\mathcal{M}^{2}- B\,e\, \left(l+\gamma+\ell-\frac{1}{2}\right)\right)\mathcal{C}_{\ell-2}+\left(Z\,e^{2}(3\,E-\mathcal{M})-\frac{B\, e^{3} Z \,(l+\gamma+\ell+\frac{3}{2})}{E+\mathcal{M}}\right)\mathcal{C}_{\ell-1}\notag\\
&+\left((2\,\ell+1)\,\gamma+\ell\, (\ell-1)-l -\frac{a\, e^{2}\, Z}{E+\mathcal{M}}-\frac{1}{2}\right)\mathcal{C}_{\ell}+\frac{e^{2}\,Z\, (1+\ell)\, (1+2\, \gamma+\ell)}{E+\mathcal{M}}\,\mathcal{C}_{\ell+1}=0 \, , \quad \ell=0,\,1,\,\ldots,\, m+2 \, .
\end{align}
From Theorem \ref{Thm4.2}, the necessary condition for the polynomial solutions of \eqref{eq5.15} is
\begin{align}\label{eq5.17}
\varepsilon=e\,B\,(m-1)\qquad \Longrightarrow\qquad E_{m}^{2}=e\,B\, \left(m+l+\gamma+\frac{3}{2}\right)+\mathcal{M}^{2}\, , \quad m=0,\,1,\,\ldots \, .
\end{align}
For the zero-degree polynomial solution $Q_{0}(r)=1$ where $\mathcal{C}_{0}=1$, $m=0$ and therefore $\ell=0,\,1,\,2$, it follows that:
\begin{align}
& \ell=0,\,1: \quad \left\{ \begin{array}{lr}
(1+2 \,l-2\, \gamma)\, (E+\mathcal{M})+4\, e^{4}\, Z^{2}\,E =0 \, ,&\mbox{(S1)}  \\ 
\notag \\
e\, B \left( l+\gamma+\frac{5}{2}\right)-(3 \,E-\mathcal{M}) \,(E+\mathcal{M})=0 \, , &\mbox{(S2)}  \notag
\end{array} \right. \\ 
\notag\\
& \ell=2: \quad  E^{2}=B\, e\, \left(l+\gamma+\frac{3}{2}\right)+ \mathcal{M}^{2} \, . \quad\mbox{(NC)} \label{eq5.18} 
\end{align}
For the first-degree polynomial solution $Q_{1}(r)=1+\mathcal{C}_{1}\, r$ where $\mathcal{C}_{0}=1$ and $\mathcal{C}_{1}\neq 0$, $m=1$ and therefore $\ell=0,\,1,\,2,\,3$, it follows that:
\begin{align}
& \ell=0: \quad  \mathcal{C}_{1}=\dfrac{(1 + 2 l - 2 \gamma)\, (E + \mathcal{M}) + 4\,e^{4} E\, Z^{2}}{2e^{2}\,Z\,(1+2\gamma)} \, , \notag\\ 
& \ell=1,\,2: \quad 
 \left\{ \begin{array}{ll}
e^{2} \,Z \Big(B\, e\, (5 + 2 l + 2 \gamma) - 2 \,(3\, E -\mathcal{M}) \,(E + \mathcal{M}) \Big) & \\
\hskip0.5true in+ 
\Big((1 + 2 \,l - 6\, \gamma) (E + \mathcal{M}) + 4 \,e^{4} \,E\, Z^{2}\Big) \, \mathcal{C}_{1}=0 \, , &\mbox{(S1)} \notag\\ \\
E^{2}- B \,e \left( l+ \gamma+\frac{3}{2}\right)-\mathcal{M}^{2}& \\
+\bigg(e^{2}\,Z\, (3\,E-\mathcal{M})-\dfrac{B \,e^{3} \,Z\, (l+ \gamma+\frac{7}{2})}{E+\mathcal{M}}\bigg) \, \mathcal{C}_{1}=0 \, ,&\mbox{(S2)} \\
\end{array} \right. \notag\\
& \ell=3: \quad  E^{2}=e\, B\, \left(l+\gamma+\frac{5}{2}\right)+ \mathcal{M}^{2} \, . \quad\mbox{(NC)} \label{eq5.19} 
\end{align}
For the second-degree polynomial solution $Q_{2}(r)=1+\mathcal{C}_{1} \, r+\mathcal{C}_{2} \,  r^{2}$ where $\mathcal{C}_{0}=1$ and $\mathcal{C}_{2}\neq 0$, $m=2$ and therefore $\ell=0,\,1,\,2,\,3,\,4$, it follows that:
\begin{align*}
& \ell=0,1:  \left\{ \begin{array}{l} 
\mathcal{C}_{1} = \dfrac{(1 + 2\, l - 2 \,\gamma)\, (E + \mathcal{M}) + 4\, e^{4}\, E\, Z^{2}}{2\, e^{2} (1 + 2\, \gamma)\, Z } \, , \\
\\
\Big(\left(l - 3\, \gamma+\frac{1}{2}\right)\, (E + \mathcal{M}) + 2\, e^{4} \,Z^{2}\,E \Big) \mathcal{C}_{1} - \Big(4 \,e^{2}\,Z \,(1 + \gamma) \Big) \, \mathcal{C}_{2} \\
\hskip0.5true in + e^{2}\,Z \,\Big(B\, e \left( l + \gamma+\frac{5}{2}\right) -(3 \,E -\mathcal{M}) \,(E + \mathcal{M}) \Big) = 0 \, ,
\end{array} \right. 
\end{align*}
\begin{align}
& \ell=2,\,3: 
 \left\{ \begin{array}{ll}
\begin{array}{l} 
(E + \mathcal{M}) \Big( E^{2} -\mathcal{M}^{2} -B\, e \left(l + \gamma+\frac{3}{2}\right) \Big) \mathcal{C}_{0} \\
\quad + e^{2}\,Z\, \Big( (3\,E -\mathcal{M})\,(E + \mathcal{M})- B\, e \left(l +\gamma+\frac{7}{2}\right)  \Big) \, \mathcal{C}_{1} \\
\quad \quad + \Big( \left(\frac{3}{2} -l +5\,\gamma\right) (E +\mathcal  M) - 2 \,e^{4}\, E \,Z^{2} \Big) \, \mathcal{C}_{2}=0 \, ,  \\
\end{array} & \quad \mbox{(S1)}  \\
\\
\begin{array}{l} 
\Big(E^{2}-\mathcal{M}^{2} - B\, e \left( l + \gamma+\frac{5}{2}\right)\Big) \, \mathcal{C}_{1} \\
\quad -\dfrac{e^{2}\,Z \,\Big( B\, e \,( l + \gamma+\frac{9}{2}) - (3\, E -\mathcal{M}) \,(E +\mathcal{M}) \Big)}{ E + \mathcal{M}} \, \mathcal{C}_{2}=0 \, ,
\end{array}  & \quad \mbox{(S2)}  \\
\end{array} \right. \notag \\
& \ell=4: \quad E^{2}= e\, B  \left(l + \gamma+\tfrac{7}{2}\right) +\mathcal{M}^{2} \, . \quad \mbox{(NC)} \label{eq5.20}
\end{align}
Higher-order polynomial solutions follows similarly.

\subsection{The Heun Equation}
Over the past two decades Heun's equation has attracted considerable attention due to its increasing number of applications in applied mathematics and theoretical physics \cite{ma2007,ex1991,ha2010,ja2008,ma1985,ci2010,ca2014,ov2012,ro1974,ro1995b}. The general second-order Heun differential equation can be written as
\begin{equation} \label{eq5.21}
\big( \alpha_{3} \, r^{3} + \alpha_{2} \, r^{2} + \alpha_{1} \, r \big) \, y''(r) + \big( \beta_{2} \, r^{2} + \beta_{1} \, r + \beta_{0} \big) \, y'(r) - \big( \tau_{1} \, r + \tau_{0} \big) \, y(r) = 0 \, ,
\end{equation}
where the canonical form \cite{ro1995b} may be obtained using the substitutions given in Table \ref{Tab1}. \\
\begin{table}[ht]
\centering
\renewcommand*{\arraystretch}{1.2}
\begin{tabular}{ R | C C C C }
i & \hspace*{1cm} 3 \hspace*{1cm} & \hspace*{1cm} 2 \hspace*{1cm} & \hspace*{1cm} 1 \hspace*{1cm} & \hspace*{1cm} 0 \hspace*{1cm} \\
\hline
\alpha_{i} & 1 & -(1-c) & c & 0 \\
\hline
\beta_{i} & & \delta+\epsilon+\gamma & - \left( c \, ( \delta + \gamma ) + \epsilon + \gamma \right) & \gamma \, c \\
\hline
\tau_{i} & & & - \alpha \, \beta & q \\
\end{tabular}
\caption{Tabulated values for parameters $\alpha_{i}$, $\beta_{i}$, and $\tau_{i}$ which would transform equation \eqref{eq5.21} to the canonical form of the Heun equation.} \label{Tab1}
\renewcommand*{\arraystretch}{1}
\end{table}\\
We immediately see that because $\alpha_{0} = 0$ and $\alpha_{1} \neq 0$ that \eqref{eq5.21} falls under Theorem \ref{Thm4.2}. By equation \eqref{eq4.5} we obtain the three-term recurrence relation
\begin{align}
\big[ \alpha_{3} \, (\ell + s - 2) (\ell + s - 1) + \beta_{2} \, (\ell + s - 1) - \tau_{1} \big] \, \mathcal{C}_{\ell-1} & + \big[ \alpha_{2} \, (\ell + s - 1) (\ell + s) + \beta_{1} \, (\ell + s) - \tau_{0}\big] \, \mathcal{C}_{\ell} \notag \\
& + \big[ \alpha_{1} \, (\ell + s) (\ell + s + 1) + \beta_{0} \, (\ell + s + 1) \big] \, \mathcal{C}_{\ell+1} = 0 \, , \label{eq5.22}
\end{align} 
for $\ell = 0,\,1,\,2,\,\ldots,\,m+1$ and $s \in \big\{ 0 \, , \, 1-(\beta_{0}/\alpha_{1}) \big\}$. We may then input this into the Mathematica\textsuperscript{\textregistered} program or work directly with the ($n+2$) equations generated by the recurrence relations \eqref{eq5.22}. Using either approach with $s=0$ will generate the results, as follows. \\

For the zero-degree ($m=0$) polynomial solution, $\mathcal{P}_{0}(r)=1$, $\ell=0,\,1$. Noting that $\mathcal{C}_{j}=0$ for all $j \neq 0$, it follows for $\ell=0$ that $\tau_{0}=0$, while for $\ell=1$, $\tau_{1}=0$.\\

For the first-degree ($m=1$) polynomial solution, $\mathcal{P}_{1}(r) = \mathcal{C}_{0} + \mathcal{C}_{1} \, r$, $\ell = 0,\,1,\,2$. It follows that
\begin{align} \label{eq5.23}
& \ell=0: \quad \mathcal{C}_{1}=\tau_{0}/\beta_{0} \, , \qquad \ell=1: \quad
 \left| \begin{array}{ll}
-\tau_{0} & \beta_{0} \\
-\tau_{1} & \beta_{1}-\tau_{0}
\end{array} \right| = 0 \, , \quad \mbox{(SC)} \qquad \ell=2: \quad \tau_{1}=\beta_{2} \, . \quad \mbox{(NC)}
\end{align}
where (SC) refers to the sufficient condition and (NC) refers to the necessary condition, and we have set $\mathcal{C}_{0} = 1$ for convenience. \\

For the second-degree $(m=2)$ polynomial solution $\mathcal{P}_{2}(r)= 1+\mathcal{C}_{1}\,r + \mathcal{C}_{2}\,r^{2}$,  $\ell=0,\,1,\,2,\,3$ yields
\begin{align*}
& \ell=0,\,1: \quad \left\{ 
\begin{array}{ll}
\mathcal{C}_{1} &= \frac{
\left|\begin{array}{ll} 
\tau_{0} & 0 \\  
\tau_{1}& 2 \left( \alpha_{1} + \beta_{0} \right)
\end{array}\right|}{\left|\begin{array}{ll} 
\beta_{0} & 0 \\ 
\beta_{1} - \tau_{0} & 2 \left( \alpha_{1} + \beta_{0} \right)
\end{array}\right|} \, , \qquad
\mathcal{C}_{2} = \frac{
\left|\begin{array}{ll} 
\beta_{0} & \tau_{0} \\
\beta_{1} - \tau_{0} & \tau_{1}
\end{array}\right|}{\left|\begin{array}{ll} 
\beta_{0} & 0 \\
\beta_{1} - \tau_{0} & 2 \left( \alpha_{1} + \beta_{0} \right)
\end{array}\right|} \, ,
\end{array} \right. 
 \end{align*}
\begin{align}
& \ell=2: \quad 
\left|\begin{array}{lll} 
-\tau_{0} & \beta_{0} & 0 \\
-\tau_{1} & \beta_{1} - \tau_{0} & 2 \left( \alpha_{1} + \beta_{0} \right) \\
0 & \beta_{2} - \tau_{1} & 2 \left( \alpha_{2} + \beta_{1} \right) - \tau_{0}
\end{array}\right|=0 \, , \quad \mbox{(SC)} \qquad
\ell=3: \quad  2 \, \alpha_{3} + 2 \, \beta_{2} - \tau_{1} = 0 \, . \quad \mbox{(NC)} \label{eq5.24}
\end{align}

For the third degree $(m=3)$ polynomial solution $\mathcal{P}_{3}(r)=1+\mathcal{C}_{1}\,r+\mathcal{C}_{2}\,r^{2}+\mathcal{C}_{3}\,r^{3}$. Noting that $\mathcal{C}_{j}=0,$ for all $j\neq 0,1,2$, we have for $\ell=0,\,1,\,2,\,3,\,4$
 
\begin{align*}
\ell=0,\,1,\,2: \quad \left\{ 
\begin{array}{ll}
\mathcal{C}_{1} &= \frac{
\left|\begin{array}{lll} 
\tau _{0} & 0 & 0 \\
\tau_{1} & 2 \left(\alpha_{1}+\beta_{0}\right) & 0 \\
0 & 2 \left(\alpha_{2}+\beta_{1}\right)-\tau _{0} & 3 \left(2 \alpha_{1}+\beta_{0}\right) \\
\end{array}\right|}{\left|\begin{array}{lll} 
\beta_{0} & 0 & 0 \\
\beta_{1}-\tau _{0} & 2 \left(\alpha_{1}+\beta_{0}\right) & 0 \\
\beta_{2}-\tau_{1} & 2 \left(\alpha_{2}+\beta_{1}\right)-\tau _{0} & 3 \left(2 \alpha_{1}+\beta_{0}\right) \\
\end{array}\right|} \, , \\
\\
\mathcal{C}_{2} &= \frac{
\left|\begin{array}{lll} 
\beta_{0} & \tau _{0} & 0 \\
\beta_{1}-\tau _{0} & \tau_{1} & 0 \\
\beta_{2}-\tau_{1} & 0 & 3 \left(2 \alpha_{1}+\beta_{0}\right) \\
\end{array}\right|}{\left|\begin{array}{lll} 
\beta_{0} & 0 & 0 \\
\beta_{1}-\tau _{0} & 2 \left(\alpha_{1}+\beta_{0}\right) & 0 \\
\beta_{2}-\tau_{1} & 2 \left(\alpha_{2}+\beta_{1}\right)-\tau _{0} & 3 \left(2 \alpha_{1}+\beta_{0}\right) \\
\end{array}\right|} \, , \\
\\
\mathcal{C}_{3} &= \frac{
\left|\begin{array}{lll} 
\beta_{0} & 0 & \tau _{0} \\
\beta_{1}-\tau _{0} & 2 \left(\alpha_{1}+\beta_{0}\right) & \tau_{1} \\
\beta_{2}-\tau_{1} & 2 \left(\alpha_{2}+\beta_{1}\right)-\tau _{0} & 0 \\
\end{array}\right|}{\left|\begin{array}{lll} 
\beta_{0} & 0 & 0 \\
\beta_{1}-\tau _{0} & 2 \left(\alpha_{1}+\beta_{0}\right) & 0 \\
\beta_{2}-\tau_{1} & 2 \left(\alpha_{2}+\beta_{1}\right)-\tau _{0} & 3 \left(2 \alpha_{1}+\beta_{0}\right) \\
\end{array}\right|} \, ,
\end{array} \right.
\end{align*}
\begin{align}
& \ell=3: \quad
\left|\begin{array}{llll} 
-\tau _{0} & \beta_{0} & 0 & 0 \\
-\tau_{1} & \beta_{1}-\tau _{0} & 2 \left(\alpha_{1}+\beta_{0}\right) & 0 \\
0 & \beta_{2}-\tau_{1} & 2 \left(\alpha_{2}+\beta_{1}\right)-\tau _{0} & 3 \left(2 \, \alpha_{1}+\beta_{0}\right) \\
0 & 0 & 2 \left(\alpha_{3}+\beta_{2}\right)-\tau_{1} & 6 \, \alpha_{2}+3 \, \beta_{1}-\tau _{0} \\
\end{array}\right|=0 \, , \quad \mbox{(SC)} \notag \\
\notag \\ 
& \ell=4: \quad 6 \, \alpha_{3} + 3 \, \beta_{2} - \tau_{1}\ = 0 \, . \quad \mbox{(NC)} \label{eq5.25}
\end{align}
This approach allows us to construct all the Heun polynomial solutions in a constructive way that is not available in the literature.

\section{Conclusions} \label{Sec6}
In the present work, we provided a simple and constructive approach to find the possible polynomial solutions of the linear differential equation \eqref{eq1.3}, along with the existence conditions. Although there may be other sophisticated approaches available in the literature for solving differential equations \cite{bl2012}, the method presented here is characterized by its constructive approach through recurrence relations. It may be adopted not only to obtain significant research results but also can be used for educational purposes. Aside from the examples discussed in this present work, the authors explored a large number of linear differential equations which appeared in mathematics and physics literature, and have found exact consistency with the results obtained by other researchers \cite{ex1991,ha2010,ja2008,ma1985,ha2011,ha2011b,ci2010,ca2014,ov2012,ro1974}.

\section*{Appendix}
\setcounter{equation}{0}
\renewcommand{\theequation}{A.\arabic{equation}}

\textbf{Proof of Theorem \ref{Thm4.1}:} For $\alpha_{0}\neq 0$ the rational functions $p(r) ={ \sum_{k=0}^{n-1} \beta_{k} \, r^{k} }/{ \sum_{k=0}^{n} \alpha_{k} \, r^{k} } $ and $q(r) = -{ \sum_{k=0}^{n-2} \tau_{k} \, r^{k} }/{ \sum_{k=0}^{n} \alpha_{k} \, r^{k} }$ are analytic at $r=0$. Consequently,  $r=0$ is an ordinary point of the differential equation given by \eqref{eq1.3}
 with a power series solution of the form
\begin{equation}\label{eqA.1}
y(r)=\sum_{\ell =0}^{\infty} \mathcal{C}_{\ell} \, r^{\ell} \, , \quad y'(r)=\sum_{\ell =1}^{\infty} \mathcal{C}_{\ell} \, \ell \, r^{\ell-1} \, , \quad y''(r)=\sum_{\ell=2}^{\infty}  \mathcal{C}_{\ell} \, \ell \, (\ell-1) \, r^{\ell-2} \, .
\end{equation}
Substituting  $y(r)$, $y'(r)$, and $y''(r)$ from \eqref{eqA.1} into \eqref{eq4.1} yields
\begin{equation}\label{eqA.2}
\sum_{k=0}^{n} \sum_{\ell=2}^{\infty} \left( \alpha_{k} \, \mathcal{C}_{\ell} \, \ell \, (\ell-1) \, r^{k+\ell-2} \right) + \sum_{k=0}^{n-1} \sum_{\ell=1}^{\infty} \left( \beta_{k} \, \mathcal{C}_{\ell} \, \ell \, r^{k+\ell-1} \right) - \sum_{k=0}^{n-2}\sum_{\ell=0}^{\infty} \left( \tau_{k} \, \mathcal{C}_{\ell} \, r^{k+\ell} \right) = 0 \, .
\end{equation}
We then shift the $\ell$-summation to arrive at
\begin{equation}\label{eqA.3}
\sum_{k=0}^{n} \sum_{\ell=2}^{\infty} \Big( \alpha_{k} \, \mathcal{C}_{\ell} \, \ell \, (\ell-1) \, r^{k+\ell-2} \Big) + \sum_{k=0}^{n-1} \sum_{\ell=2}^{\infty} \Big( \beta_{k} \, \mathcal{C}_{\ell-1} \, (\ell-1) \, r^{k+\ell-2} \Big) - \sum_{k=0}^{n-2} \sum_{\ell=2}^{\infty} \Big( \tau_{k} \, \mathcal{C}_{\ell-2} \, r^{k+\ell-2} \Big) = 0 \, .
\end{equation}
Now we group these into common degrees of $r$
\begin{align}\label{eqA.4}
\sum_{\ell=2}^{\infty} & \Big[ \alpha_{0} \, \mathcal{C}_{\ell} \, \ell \, (\ell-1) + \beta_{0} \, \mathcal{C}_{\ell-1} \, (\ell-1) - \tau_{0} \, \mathcal{C}_{\ell-2} \Big] \,  r^{\ell-2} \notag \\
&+ \sum_{\ell=2}^{\infty} \Big[ \alpha_{1} \, \mathcal{C}_{\ell} \, \ell \, (\ell-1) + \beta_{1} \, \mathcal{C}_{\ell-1} \, (\ell-1) - \tau_{1} \, \mathcal{C}_{\ell-2} \Big] \, r^{\ell-1} + \sum_{\ell=2}^{\infty} \Big[ \alpha_{2} \, \mathcal{C}_{\ell} \, \ell \, (\ell-1) + \beta_{2} \, \mathcal{C}_{\ell-1} \, (\ell-1) - \tau_{2	} \, \mathcal{C}_{\ell-2} \Big] \, r^{\ell} \notag \\
&+ \cdots + \sum_{\ell=2}^{\infty} \Big[ \alpha_{n-2} \, \mathcal{C}_{\ell} \, \ell \, (\ell-1) + \beta_{n-2} \, \mathcal{C}_{\ell-1} \, (\ell-1) - \tau_{n-2} \, \mathcal{C}_{\ell-2} \Big] \, r^{\ell+n-4} \notag \\
& \qquad + \sum_{\ell=2}^{\infty} \Big[ \alpha_{n-1} \, \mathcal{C}_{\ell} \, \ell \, (\ell-1) + \beta_{n-1} \, \mathcal{C}_{\ell-1} \, (\ell-1) \Big] \, r^{\ell+n-3} + \sum_{\ell=2}^{\infty} \Big[ \alpha_{n} \, \mathcal{C}_{\ell} \, \ell \, (\ell-1) \Big] \, r^{\ell+n-2} = 0 \, .
\end{align}
Then, shifting the individual sums obtain a common degree of $r$ results in
\begin{align}\label{eqA.5}
& \sum_{\ell+n=2}^{\infty} \Big[ \alpha_{0} \, \mathcal{C}_{\ell+n} \, (\ell+n) \, (\ell+n-1) + \beta_{0} \, \mathcal{C}_{\ell+n-1} \, (\ell+n-1) - \tau_{0} \, \mathcal{C}_{\ell+n-2} \Big] \, r^{\ell+n-2} \notag \\
& \qquad + \sum_{\ell+n-1=2}^{\infty} \Big[ \alpha_{1} \, \mathcal{C}_{\ell+n-1} \, (\ell+n-1) \, (\ell+n-2) + \beta_{1} \, \mathcal{C}_{\ell+n-2} \, (\ell+n-2) - \tau_{1} \, \mathcal{C}_{\ell+n-3} \Big] \, r^{\ell+n-2} \notag \\
& \qquad + \sum_{\ell+n-2=2}^{\infty} \Big[ \alpha_{2} \, \mathcal{C}_{\ell+n-2} \, (\ell+n-2) \, (\ell+n-3) + \beta_{2} \, \mathcal{C}_{\ell+n-3} \, (\ell+n-3) - \tau_{2} \, \mathcal{C}_{\ell+n-4} \Big] \, r^{\ell+n-2} \notag \\
+ \cdots + & \sum_{\ell+2=2}^{\infty} \Big[ \alpha_{n-2} \, \mathcal{C}_{\ell+2} \, (\ell+2) \, (\ell+1) + \beta_{n-2} \, \mathcal{C}_{\ell+1} \, (\ell+1) - \tau_{n-2} \, \mathcal{C}_{\ell} \Big] \, r^{\ell+n-2} \notag \\
& \qquad + \sum_{\ell+1=2}^{\infty} \Big[ \alpha_{n-1} \, \mathcal{C}_{\ell+1} \, (\ell+1) \, \ell + \beta_{n-1} \, \mathcal{C}_{\ell} \, \ell \Big] \, r^{\ell+n-2}+ \sum_{\ell=2}^{\infty} \Big[ \alpha_{n} \, \mathcal{C}_{\ell} \, \ell \, (\ell-1) \Big]\,  r^{\ell+n-2} = 0 \, .
\end{align}
To enforce polynomial solutions of degree $m$, we may modify the summations to go to an upper limit of $m$ rather than $\infty$. We wish to shift our bottom index by the degree of the polynomial on the $\tau$ terms, thus a shift of $\ell = \ell-n+2$ is required. The above formula exposes the $(n+1)$-term recurrence relation which may be grouped into common $\mathcal{C}_{j}$ as
\begin{align}\label{eqA.6}
\Big[ \alpha_{0} \, (\ell+2) \, (\ell+1) \Big] \, \mathcal{C}_{\ell+2} &+  \Big[ \alpha_{1} \, (\ell+1) \, \ell + \beta_{0} \, (\ell+1) \Big] \, \mathcal{C}_{\ell+1} +  \Big[ \alpha_{2} \, \ell \, (\ell-1) + \beta_{1} \, \ell - \tau_{0} \Big] \, \mathcal{C}_{\ell} \notag \\
+ & \cdots + \Big[  \alpha_{n} \, (\ell-n+2) \, (\ell-n+1) + \beta_{n-1} \, (\ell-n+2) - \tau_{n-2} \Big] \, \mathcal{C}_{\ell-n+2} = 0 \, .
\end{align}
Finally, the $m^{th}$-degree polynomial solution is given by
\begin{equation}\label{eqA.7}
\sum_{j=0}^{n} \Big[ \alpha_{n-j} \, (\ell-n+j+2) \, (\ell-n+j+1) + \beta_{n-j-1} \, (\ell-n+j+2) - \tau_{n-j-2} \Big] \, \mathcal{C}_{\ell-n+j+2} = 0 \, .
\end{equation} 
where $\ell = 0, \, 1, \, 2, \, \ldots, \, m+n-2$, $n \geq 2$, and $\tau_{i} = \beta_{i} = 0$ if $i<0$.\\

\textbf{Proof of Theorem \ref{Thm4.2}:} For $\alpha_{0}= 0$ and $\alpha_{1}\neq 0$ the rational functions $ p(r) = \sum_{k=0}^{n-1} \beta_{k} \, r^{k}/\sum_{k=1}^{n} \alpha_{k} \, r^{k}$ and $q(r) = - \sum_{k=0}^{n-2} \tau_{k} \, r^{k}/{ \sum_{k=1}^{n} \alpha_{k} \, r^{k} } $,
are non-analytic at $r=0$. We see that 
\begin{equation}\label{eqA.8}
\lim_{r \rightarrow 0} \left[ r \, \frac{\sum_{k=0}^{n-1} \beta_{k} \, r^{k}}{\sum_{k=1}^{n} \alpha_{k} \, r^{k}}\right] = \frac{\beta_{0}}{\alpha_{1}} \, , \qquad \lim_{r \rightarrow 0} \left[- r^{2} \, \frac{\sum_{k=0}^{n-2} \tau_{k} \, r^{k}}{\sum_{k=1}^{n} \alpha_{k} \, r^{k}}\right] = 0 \, .
\end{equation}
Thus the indicial roots are given by: $s(s-1) + ({\beta_{0}}/{\alpha_{1}})s = 0$, so $s \in \big\{ 0 \, , \, 1-({\beta_{0}}/{\alpha_{1}}) \big\}$ and $r=0$ is a regular singular point of the differential equation \eqref{eq4.4}. 
The method of Frobenius assumes a formal series solution of \eqref{eq4.4} of the form
\begin{equation}\label{eqA.9}
y = \sum_{\ell=0}^{\infty} \mathcal{C}_{\ell} \, r^{\ell+s} \, , \qquad y'(r) = \sum_{\ell=0}^{\infty} \mathcal{C}_{\ell} \, (\ell+s) \, r^{\ell+s-1} \, , \qquad y''(r) = \sum_{\ell=0}^{\infty} \mathcal{C}_{\ell} \, (\ell+s)(\ell+s-1) \, r^{\ell+s-2} \, .
\end{equation}
Substituting  $y(r)$, $y'(r)$, and $y''(r)$ from \eqref{eqA.9} into \eqref{eq4.4} yields
\begin{equation}\label{eqA.10}
\sum_{k=1}^{n} \sum_{\ell=0}^{\infty} \alpha_{k} \, \mathcal{C}_{\ell} \, (\ell+s) (\ell+s-1) \, r^{k+\ell+s-2} + \sum_{k=0}^{n-1} \sum_{\ell=0}^{\infty} \beta_{k} \, \mathcal{C}_{\ell} \, (\ell+s) \, r^{k+\ell+s-1} 
- \sum_{k=0}^{n-2} \sum_{\ell=0}^{\infty} \tau_{k} \, \mathcal{C}_{\ell} \, r^{k+\ell+s} = 0 \, .
\end{equation}
We then group these into a common degree of $r$:
\begin{align} \label{eqA.11}
\sum_{\ell=0}^{\infty}& \Big[ \alpha_{1} \, \mathcal{C}_{\ell} \, (\ell+s) (\ell+s-1) + \beta_{0} \, \mathcal{C}_{\ell} \, (\ell+s) \Big] \, r^{\ell+s-1} + \sum_{\ell=0}^{\infty} \Big[ \alpha_{2} \, \mathcal{C}_{\ell} \, (\ell+s) (\ell+s-1) + \beta_{1} \, \mathcal{C}_{\ell} \, (\ell+s) - \tau_{0} \, \mathcal{C}_{\ell} \Big] \, r^{\ell+s} \notag\\
+ & \cdots +  \sum_{\ell=0}^{\infty} \Big[ \alpha_{n} \, \mathcal{C}_{\ell} \, (\ell+s) (\ell+s-1) + \beta_{n-1} \, \mathcal{C}_{\ell} \, (\ell+s) - \tau_{n-2} \, \mathcal{C}_{\ell} \Big] \, r^{n+\ell+s-2} = 0 \, .
\end{align}
Now we shift the $\ell$-summation to obtain $r^{n+\ell+s-2}$ in each term
\begin{align}\label{eqA.12}
 \sum_{n+\ell-1=0}^{\infty}& \Big[ \alpha_{1} \, \mathcal{C}_{n+\ell-1} \, (n+\ell+s-1) (n+\ell+s-2) + \beta_{0} \, \mathcal{C}_{n+\ell-1} \, (n+\ell+s-1) \Big] \, r^{n+\ell+s-2} \notag\\
 &+ \sum_{n+\ell-2=0}^{\infty} \Big[ \alpha_{2} \, \mathcal{C}_{n+\ell-2} \, (n+\ell+s-2) (n+\ell+s-3) + \beta_{1} \, \mathcal{C}_{n+\ell-2} \, (n+\ell+s-2) - \tau_{0} \, \mathcal{C}_{\ell} \Big] \, r^{n+\ell+s-2} \notag\\
&+  \cdots + \sum_{\ell=0}^{\infty} \Big[ \alpha_{n} \, \mathcal{C}_{\ell} \, (\ell+s) (\ell+s-1) + \beta_{n-1} \, \mathcal{C}_{\ell} \, (\ell+s) - \tau_{n-2} \, \mathcal{C}_{n+\ell-2} \Big] \, r^{n+\ell+s-2} = 0 \, .
\end{align}
To enforce polynomial solutions of degree $m$, we may modify the summations to go to an upper limit of $m$ rather than $\infty$. We wish to shift our bottom index by the degree of the polynomial on the $\tau$ terms, thus a shift of $\ell = \ell-n+2$ is required. The above formula exposes the $n$-term recurrence relation which may be grouped into common $\mathcal{C}_{j}$ as
\begin{align}\label{eqA.13}
\Big[ \alpha_{1} \, (\ell+s+1) \, (\ell+s) &+ \beta_{0} \, (\ell+s+1) \Big] \, \mathcal{C}_{\ell+1}+  \Big[ \alpha_{2} \, (\ell+s) (\ell+s-1) + \beta_{1} \, (\ell+s) - \tau_{0} \Big] \, \mathcal{C}_{\ell} \notag \\
&+  \cdots +  \Big[  \alpha_{n} \, (\ell-n+s+2) \, (\ell-n+s+1) + \beta_{n-1} \, (\ell-n+s+2) - \tau_{n-2} \Big] \, \mathcal{C}_{\ell-n-2} = 0 \, .
\end{align}
Thus, the $m^{th}$-degree polynomial solution is given by
\begin{equation}\label{eqA.14}
\sum_{j=0}^{n-1} \Big[ \alpha_{n-j} \, (\ell-n+j+s+2) (\ell-n+j+s+1) + \beta_{n-j-1} \, (\ell-n+j+s+2) - \tau_{n-j-2} \Big] \, \mathcal{C}_{\ell-n+j+2} = 0 \, .
\end{equation} 
where $\ell = 0, \, 1, \, 2, \, \ldots, \, m+n-2$, $n \geq 2$, and $\tau_{i} = 0$ if $i<0$.

\section*{Acknowledgments}
Partial financial support of this work, under Grant No. GP249507 from the Natural Sciences and Engineering Research Council of Canada, is gratefully acknowledged. 

\bibliographystyle{elsarticle-harv}

\end{document}